\documentclass[12pt]{amsart}
\usepackage{amsmath, amsthm, amssymb}
\usepackage{graphicx}
\usepackage{mathbbol}
\usepackage{txfonts}
\usepackage[usenames]{color}

\newtheorem{thm}{Theorem}[section]
\newcommand{\bt}{\begin{thm}}
\newcommand{\et}{\end{thm}}

\newtheorem{ex}[thm]{Example}

\newtheorem{cor}[thm]{Corollary}   %remember switch all {coro} to {cor}
\newcommand{\bc}{\begin{cor}}
\newcommand{\ec}{\end{cor}}

\newtheorem{lem}[thm]{Lemma}   %remember to switch all {lem} to {lem}
\newcommand{\bl}{\begin{lem}}
\newcommand{\el}{\end{lem}}

\newtheorem{prop}[thm]{Proposition}
\newcommand{\bp}{\begin{prop}}
\newcommand{\ep}{\end{prop}}

\newtheorem{defn}[thm]{Definition}
\newcommand{\bd}{\begin{defn}}    % This produces an error????    
\newcommand{\ed}{\end{defn}}

\newtheorem{rmrk}[thm]{Remark}   %remember to switch all {rmrk} to {rmrk}

\newcommand{\br}{\begin{rmrk}}
\newcommand{\er}{\end{rmrk}}

\newcommand{\weaklyto}{\to}

\newtheorem{example}[thm]{Example}

\newcommand{\GHto}{\stackrel { \textrm{GH}}{\longrightarrow} }
\newcommand{\Fto}{\stackrel {\mathcal{F}}{\longrightarrow} }

\newcommand{\be}{\begin{equation}}

 \newcommand{\ee}{\end{equation}}

\newcommand{\N}{\mathbb{N}}

\newcommand{\R}{\mathbb{R}}
\newcommand{\One}{{\bf \rm{1}}}

\newcommand{\Z}{\mathbb{Z}}

\newcommand{\diam}{\operatorname{Diam}}

\newcommand{\Fm}{{\mathcal F}}

\newcommand{\set}{{\rm{set}}}

\newcommand{\Lip}{\operatorname{Lip}}

\newcommand{\mass}{{\mathbf M}}

\newcommand{\curr}{{\mathbf M}}         %metric current%
\newcommand{\rectcurr}{{\mathcal R}}    %metric rect. current%  
\newcommand{\intrectcurr}{{\mathcal I}} %metric int. rect. current%
\newcommand{\intcurr}{{\mathbf I}}      %metric integral current%

\newcommand{\vol}{\operatorname{Vol}}
\newcommand{\nmass}{\mathbf N}

\newcommand{\rstr}{\:\mbox{\rule{0.1ex}{1.2ex}\rule{1.1ex}{0.1ex}}\:}
\newcommand{\bdry}{\partial}
\newcommand{\spt}{\operatorname{spt}}

\begin{document}

\title{Intrinsic Flat Arzela-Ascoli Theorems}

\author{Christina Sormani}
\thanks{This material is based upon work supported by the National Science Foundation under Grant No. 0932078 000, while the author was serving as a Visiting Research Professor at the {\em Mathematical Sciences Research Institute} in Berkeley, California, during the Fall of 2013.   The author's research was also funded by an individual research grant NSF-DMS-1309360 and a PSC-CUNY Research Grant.}

\address{CUNY Graduate Center and Lehman College}
\email{sormanic@member.ams.org}
%\date{August 2012}

\keywords{}

%49Q15 (Geometric measure and integration theory, integral and normal currents)

%\subjclass[2000]{49Q15}

\begin{abstract}
One of the most powerful theorems in metric geometry is the
Arzela-Ascoli Theorem which provides a continuous limit for
sequences of equicontinuous functions between two compact spaces.  This theorem has been extended by Gromov and 
Grove-Petersen to sequences of functions with varying
domains and ranges where the domains and the ranges respectively
converge in the Gromov-Hausdorff sense to compact limit spaces.
However such a powerful theorem does not hold when the domains and ranges only converge in the intrinsic flat sense due to the possible disappearance of points in the limit.

In this paper two Arzela-Ascoli Theorems are proven for intrinsic flat
converging sequences of manifolds: one for
uniformly Lipschitz functions with fixed range whose domains are converging
in the intrinsic flat sense, and one for sequences of uniformly local isometries
between spaces which are converging in the intrinsic flat sense.   A basic Bolzano-Weierstrass Theorem is proven
for sequences of
points in such sequences of spaces.   In addition it is proven that when a sequence of manifolds has a precompact intrinsic flat limit then the metric completion of
the limit is the Gromov-Hausdorff limit of regions within those manifolds.
Applications and suggested applications 
of these results are described in the final section
of this paper.
\end{abstract}

\maketitle

%\noindent{\bf Recall that v1 of the manuscript said:}

%This is a preliminary preprint covering everything presented
%at MIT on Feb 15, 2011 as well as additional material added
%through August 2012.   Many more results still need to be texed.
%Everything stated here has been proven
%in complete detail here and may be applied by those who need
%the results. 

\section{Introduction}

When studying sequences of Riemannian manifolds, one may use a variety of notions of convergence from $C^{k,\alpha}$ smooth convergence to Gromov-Hausdorff convergence as metric spaces.   One needs to understand whether points and balls in the sequences converge to points and balls in limit spaces.   So one proves Bolzano-Weierstrass theorems to produce converging subsequences of points.   One needs to understand the limits of functions on these spaces and local isometries between these spaces.   So one proves Arzela-Ascoli theorems for sequences of uniformly Lipschitz functions between converging spaces.   Such theorems have been proven for Gromov-Hausdorff convergence by Gromov and by Grove-Petersen
\cite{Gromov-metric} \cite{Gromov-poly} \cite{Grove-Petersen}.   They have been applied in these works as well as that of Cheeger-Colding,
Cheeger-Naber, the author, Wei, and numerous other papers including Perelman's solution of the Poincare Conjecture (c.f. \cite{ChCo-PartI} \cite{Cheeger-Naber-1} 
\cite{Sor-cosmos} \cite{SorWei1} and \cite{Perelman-flow-surgery}).

There are many questions concerning Riemannian manifolds which
cannot be addressed using these relatively strong notions of convergence.   The intrinsic flat convergence is a more flexible
notion allowing a larger class of sequences of manifolds to converge.   Gromov
has proposed that this notion would be natural to study questions
arising in \cite{Gromov-Hilbert-I}.   Lakzian has applied
intrinsic flat convergence to prove continuity of Ricci flow 
through a singularity \cite{Lakzian-Cont-Ricci}.
Dan Lee and the author have
shown intrinsic flat convergence is well adapted to questions 
arising in General Relativity \cite{LeeSormani1}.   
Additional applications of intrinsic flat convergence are 
described in the final section of this paper relating to work of Basilio,
Burago, Ivanov, Ding, Fukaya, Gromov, Huang, Lakzian, LeFloch, Lee, Munn, Portegies, Sinaei and Wei \cite{Basilio-Sormani-1} \cite{Burago-Ivanov-Vol-Tori} \cite{Ding-heat} \cite{Fukaya-1987} \cite{Gromov-Hilbert-I} \cite{Huang-Lee-Graph} \cite{LeFloch-Sormani-1} \cite{Munn-F=GH} \cite{Portegies-F-evalue} \cite{Sinaei-Harmonic} \cite{SorWei1}.

The flexibility of intrinsic flat convergence is that it allows points
to disappear in the limit.   As a consequence, many of the techniques
used to study Gromov-Hausdorff limits including the Arzela-Ascoli Theorem fail to hold when the domains and ranges of the functions
only converge in the intrinsic flat sense (c.f. Remark~\ref{no-Arz-Asc-IF-to-IF}) .    In this paper 
additional hypothesis are provided to produce two Arzela-Ascoli Theorems
[Theorems~\ref{Flat-Arz-Asc} and~\ref{Arz-Asc-Unif-Local-Isom}] as well as a basic Bolzano-Weierstrass Theorem [Theorem~\ref{B-W-BASIC}].   
A new relationship between Gromov-Hausdorff and intrinsic
flat convergence is also proven [Theorem~\ref{flat-to-GH}].   Direct applications of these
theorems are described in the final section of this paper.   

Intrinsic flat convergence was introduced by Wenger and the author in \cite{SorWen2} building upon work of
Ambrosio-Kirchheim in \cite{AK}.   It is defined for oriented Riemannian manifolds, $M^m_j$ with boundary such that
\be\label{V-A}
\vol(M_j)\le V_j \textrm{ and }\vol(\partial M_j)\le A_j.
\ee
The limit spaces obtained under this convergence are called integral current spaces.  They are either countably 
$\mathcal{H}^m$ rectifiable metric spaces of the same dimension as the
sequence or possibly the ${\bf{0}}$ space.   Intrinsic flat limits may exist for
sequences of manifolds with no Gromov-Hausdorff limit \cite{SorWen2}.  

When 
there is a Gromov-Hausdorff limit, $M_j \GHto Y$, and one has uniform bounds on volume and boundary volume, 
\be\label{V-A-unif}
\vol(M_j)\le V_0 \textrm{ and }\vol(\partial M_j)\le A_0,
\ee
then a subsequence has an intrinsic flat limit, $M_{j_i}\Fto X$ where $X\subset Y$ with the restricted distance, $d_X=d_Y$ \cite{SorWen2}.   It is possible that $X$ is the ${\bf{0}}$ space or a strict subset of $Y$
either because the sequence is collapsing or due to cancellation (see examples in \cite{SorWen2}).   See Section~\ref{sect-background}
for a review.   

This paper focuses on sequences of oriented Riemannian manifolds, $M_j^m$
satisfying (\ref{V-A}), or more generally integral current spaces
satisfying a smiliar condition, which 
converge in the intrinsic flat sense. 
The paper begins with the definition of converging and disappearing sequences of points [Definitions~\ref{point-conv} and~\ref{point-Cauchy}] and a proof that
diameter is lower semicontinuous [Theorem~\ref{diam-semi}].   
Viewing balls
within integral current spaces as integral current spaces themselves [Lemma~\ref{lem-ball}]
it is proven that, for almost every radius, balls around converging points have subsequences which converge to balls about their limit points
[Lemma~\ref{balls-converge}].     The necessity of taking a subsequence is
shown in Example~\ref{ex-balls-converge}.
If a sequence of points disappears, the 
balls of small radius about those points converge to the $\bf{0}$ space [Lemma~\ref{balls-converge}].   Lemma~\ref{rescaling} examines how the intrinsic flat distance may be estimated when the spaces are rescaled.   Although technical, these lemmas are key
steps in the subsequent theorems.
  
It is shown in Theorem~\ref{flat-to-GH} that if Riemannian manifolds $M_i$ converge in the intrinsic flat sense to a nonzero precompact limit space, $M$, 
then there are open submanifolds $N_i \subset M_i$ such that 
$N_i \GHto M$.  This theorem and Remark~\ref{flat-to-GH-r} also describe the volumes of these submanifolds as well as what happens when $M_i$ are integral current spaces.   Section~\ref{sect-flat-to-GH} also contains a few
related open questions within remarks concerning possible extensions and applications of this theorem.

Theorem~\ref{Flat-Arz-Asc} is the first Intrinsic Flat Arzela Ascoli Theorem.   It
states that if a sequence of functions,
$F_i: M_i\to W$ where $M_i \Fto M_\infty$ and $W$ is compact and
$\Lip(F_i)\le K$, then there is a converging subsequence $F_i \to F_\infty$
where $F_\infty: M_\infty \to W$ also has $\Lip(F_\infty)\le K$.   A
precise description as to exactly how $F_i \to F_\infty$ is given.   
Remarks~\ref{no-Arz-Asc-IF-to-IF} and~\ref{Arz-Asc-IF-to-GH} concern 
the impossibility and possibility of extending this theorem 
to allow the ranges to
converge in the intrinsic flat and Gromov-Hausdorff sense respectively.

Theorem~\ref{B-W-BASIC} is an Intrinsic Flat Bolzano-Weierstrass theorem
for points $p_i\in M_i$ such that $M_i\Fto M_\infty$.
Since it is known that points may disappear in the limit [Remark~\ref{rmrk-disappear}], it is necessary to add a condition to obtain a subsequence with a limit point $p_\infty$.   In Theorem~\ref{B-W-BASIC}, the
extra condition is that for almost every sufficiently small radius there is a uniform bound on the intrinsic flat distance between the balls about $p_i$
and $\bf{0}$.   Remark~\ref{rmrk-ball-bound} discusses how one can
obtain such a bound when needed.

Theorem~\ref{Arz-Asc-Unif-Local-Isom}
is the second Intrinsic Flat Arzela-Ascoli Theorem proven here.   In this theorem the domains and ranges of the functions converge in the intrinsic flat sense and have uniform upper bounds as in (\ref{V-A-unif}).
  The functions are assumed to be local isometries which are isometries on balls of fixed
radius.   It is shown that a subsequence of the functions converges to
a limit function which is also a local isometry.   If the functions are surjective,
then so is the limit.   The case where the limit spaces are possibly the $\bf{0}$
space is also considered.     Remark~\ref{r-biLip} discusses a possible extension of this
theorem to uniformly locally bi-Lipschitz functions or more simply uniformly bi-Lipschitz functions.   Remark~\ref{r-AA-ULI} discusses the necessity of
various conditions in Theorem~\ref{Arz-Asc-Unif-Local-Isom}.

In Section~\ref{sect-example} an example is presenting showing how these theorems can be
applied to prove certain sequences of Riemannian manifolds have no
intrinsic flat limit.  Additional applications to construct examples which
do have specific limits will appear in joint work with Basilio \cite{Basilio-Sormani-1}.

Section~\ref{sect-Appl} includes remarks describing the possible 
additional applications
of the various theorems in this paper.   In
particular one may be able to apply Theorem~\ref{Arz-Asc-Unif-Local-Isom}
to answer a question
posed by Gromov in \cite{Gromov-Hilbert-I} concerning the intrinsic flat
limits of tori whose universal covers have almost maximal volume growth in
the sense described by Burago-Ivanov in \cite{Burago-Ivanov-Vol-Tori}.
See Remark~\ref{r-AA-BIG}.    Additional possible
applications of Theorem~\ref{Arz-Asc-Unif-Local-Isom} to extend work of
the author with Wei are described in Remarks~\ref{r-Wei-1} and~\ref{r-Wei-2}.
It may also be possible to apply Theorem~\ref{Flat-Arz-Asc} to study the limits of harmonic functions, eigenfunctions and heat kernels.  See Remark~\ref{r-funct}.   In Remark~\ref{r-Ricci}, it is described how one may be
able to apply Theorem~\ref{BW-BASIC} to prove that the intrinsic flat 
and Gromov-Hausdorff limits of Riemannian manifolds with uniform lower
Ricci curvature bounds agree extending a theorem of the author with Wenger
in \cite{SorWen1}.   Finally there are three remarks discussing how various theorems in the paper may be applied to a variety
of questions and conjectures related to questions in General Relativity.

The author would like to thank Blaine Lawson (SUNYSB) for suggesting that the basic properties of
intrinsic flat convergence should appear in their own paper 
separate from the
more technical theorems involving the Gromov Filling Volume
which appear in \cite{Sormani-properties}\footnote{\cite{Sormani-properties} is under revision and all theorems with new proofs appearing here will be removed from that paper.}.    The author is
also indebted to
Jacobus Portegies (NYU) and Raquel Perales (SUNYSB) for their careful reading of the first version of this
paper and their extensive feedback.  

%%%%%%%%%%%%%%%%%%

\section{Background}\label{sect-background}

Here the key definitions and theorems applied in this paper
are reviewed.
Please keep in mind that this is by no means a complete introduction to
Gromov-Hausdorff convergence and Intrinsic Flat convergence.   
Only the notions that are applied in this paper are reviewed.   In fact, the primary reason for combining Theorems~\ref{flat-to-GH}, ~\ref{Flat-Arz-Asc}, ~\ref{B-W-BASIC}, and ~\ref{Arz-Asc-Unif-Local-Isom} together into this paper is because these four theorems can be proven using the same background material.   Other related theorems appearing in \cite{Sormani-properties} all
require additional results of Gromov and Ambrosio-Kirchheim.

Those who have already studied the notion of Intrinsic Flat convergence in the initial paper by the author with Wenger \cite{SorWen2}, should still review Subsections~\ref{subsect-GH}, ~\ref{subsect-Slicing} and~\ref{subsect-Iballs} which cover material not presented there.   Those who have never studied Gromov-Hausdorff or Intrinsic Flat convergence will find
the entire background section useful as a very brief but self contained introduction to the subjects.   As the author sees no reason to restate theorems, definitions and remarks, some of these statements have been repeated exactly as stated in prior background sections written by the author elsewhere.

\subsection{A Review of Gromov-Hausdorff Convergence}\label{subsect-GH}

Throughout this paper, Gromov's definition of an isometric embedding will be used:

\begin{defn}
A map $\varphi: X \to Y$ between metric spaces, $(X, d_X)$ and $(Y, d_Y)$,
is an isometric embedding iff it is distance preserving:
\be
d_Y(\varphi(x_1), \varphi(x_2)) = d_X(x_1, x_2) \qquad \forall x_1, x_2 \in X.
\ee
\end{defn}

Observe that this does not agree with the Riemannian notion of an isometric embedding.

The following is one of the more beautiful definitions of the Gromov-Hausdorff distance:

\begin{defn}[Gromov]\label{defn-GH} 
The Gromov-Hausdorff distance between two 
compact metric spaces $\left(X, d_X\right)$ and $\left(Y, d_Y\right)$
is defined as
\be \label{eqn-GH-def}
d_{GH}\left(X,Y\right) := \inf  \, d^Z_H\left(\varphi\left(X\right), \psi\left(Y\right)\right)
\ee
where $Z$ is a complete metric space, and $\varphi: X \to Z$ and $\psi:Y\to Z$ are
isometric embeddings and where the Hausdorff distance in $Z$ is defined as
\be
d_{H}^Z\left(A,B\right) = \inf\{ \epsilon>0: A \subset T_\epsilon\left(B\right) \textrm{ and } B \subset T_\epsilon\left(A\right)\}.
\ee
\end{defn}

Gromov proved that this is indeed a distance on compact metric spaces
in the sense that $d_{GH}\left(X,Y\right)=0$
iff there is an isometry between $X$ and $Y$ in \cite{Gromov-metric}.   
Gromov proved the following embedding theorem in \cite{Gromov-poly}:

\begin{thm}[Gromov] \label{Gromov-Z}
If a sequence of compact metric spaces, $X_j$,
converges in the Gromov-Hausdorff sense to a compact metric space $X_\infty$, 
\be
X_j \GHto X_\infty
\ee
then in fact there is a compact metric space, $Z$, and isometric embeddings $\varphi_j: X_j \to Z$ for $j\in \{1,2,...,\infty\}$ such that
\be
d_H^Z\left(\varphi_j(X_j),\varphi_\infty(X_\infty)\right) \to 0.
\ee
\end{thm}

This theorem allows one to define converging sequences of
points: 

\begin{defn}\label{Gromov-points}
One says that $x_j \in X_j$ converges to $x_\infty \in X_\infty$,
if there is a common space $Z$ as in Theorem~\ref{Gromov-Z}
such that $\varphi_j(x_j) \to \varphi_\infty(x)$ as points in
$Z$.  If one discusses the limits of multiple sequences of points
then one uses a common $Z$ and the same collection of $\varphi_j$
to determine the convergence.
This avoids difficulties arising from isometries in the limit space.
Then one immediately has
\be
\lim_{j\to\infty} d_{X_j}(x_j, x_j')=d_{X_\infty}(x_\infty, x_\infty')
\ee
whenever $x_j \to x_\infty$ and $x_j'\to x_\infty'$ via
a common $Z$.
\end{defn}

One can apply Theorem~\ref{Gromov-Z} to see that for any $x_\infty\in X_\infty$ there exists $x_j\in X_j$ converging to $x_\infty$ in this sense.   Also
observe that whenever $x_j$ converges to $x_\infty$ in this sense, 
\be
d_{GH}\left(B(x_j,r), B(x_\infty, r)  \right)\le d_H^Z\left(B(\varphi_j(x_j),r), B(\varphi_\infty(x_\infty),r)\right) \to 0 \qquad \forall r>0
\ee
if one views the balls $B(x_j,r)\subset X_j$ as metric spaces endowed with
the restricted metric, $d_{X_j}$, from $X_j$.   See the appendix of joint work of the author with Wei \cite{SorWei2} for a theorem concerning the induced length metrics.   
Theorem~\ref{Gromov-Z} also implies the following basic Bolzano-Weierstrass Theorem:

\begin{thm}[Gromov] \label{Gromov-B-W}
Given compact metric spaces, $X_j \GHto X_\infty$, and $x_j\in X_j$ then a 
subsequence also denoted $x_j$ 
converges to a point $x_\infty\in X_\infty$ in the sense described above.
\end{thm}

In particular, one sees that
\be \label{GH-diam}
X_j \to X_\infty \,\, \implies \,\,\lim_{j\to \infty}\diam(X_j)=\diam(X_\infty).
\ee

Gromov's embedding theorem can also be applied in combination with other extension theorems to obtain the following Gromov-Hausdorff Arzela-Ascoli Theorem.   See also the appendix of a paper of Grove-Petersen \cite{Grove-Petersen} for a detailed proof and prior work of the author for a more general statement \cite{Sor-cosmos}.

\begin{thm}[Gromov] \label{Gromov-Arz-Asc}
Given compact metric spaces $X_j \GHto X_\infty$ and $Y_j \to Y_\infty$
and equicontinuous functions $f_j: X_j \to Y_j$ in the sense that
\be
\forall \epsilon>0 \,\,\exists \delta_\epsilon>0\textrm{ such that }
d_{X_j}(x,x')< \delta_\epsilon \,\implies \,
d_{Y_j}(f_j(x), f_j(x'))\le \epsilon.
\ee
then there exists a subsequence, also denoted $f_j: X_j \to Y_j$
which converges to a continuous function $f_\infty: X_\infty \to Y_\infty$
in the sense that there exists common compact metric spaces $Z, W,$
and isometric embeddings $\varphi_j: X_j \to Z$, $\psi_j: Y_j \to W$
such that
\be
\lim_{j\to \infty} \psi_j(f_j(x_j)) = \psi_\infty(f_\infty(x_\infty))
 \textrm{ whenever } \lim_{j\to \infty}\varphi_j(x_j)=\varphi_\infty(x_\infty). 
\ee
Furthermore, if $\Lip(f_j) \le K$ then $\Lip(f_\infty)\le K$.
\end{thm}

Observe in particular that if $y_j, y'_j\in Y_j$ converge to 
$y_\infty, y'_\infty \in Y_\infty$, where $Y_j \GHto Y_\infty$, and
$\gamma_j:[0,1]\to Y_j$ are minimizing geodesics from $y_j$ to $y'_j$,
then a subsequence converges to $\gamma_\infty: [0,1]\to Y_\infty$
which one can then show is a minimizing geodesic between
$x_\infty$ and $x'_\infty$.   Thus geodesic metric spaces converge to
geodesic metric spaces.

All these theorems are key ingredients in the many important works applying Gromov-Hausdorff convergence to better understand Riemannian Geometry.   
See the classic textbook of Burago-Burago-Ivanov \cite{BBI}, the work of Cheeger-Colding \cite{ChCo-PartI} and the work of the author with Wei \cite{SorWei1}.

In this paper these theorems are extended, as far as possible, in the setting where one only has intrinsic flat convergence.   Of course it is known that these theorems do not hold in their full strength in the setting where sequences of Riemannian manifolds are converging in the intrinsic flat sense.   Examples in joint work of the author with Wenger in \cite{SorWen2} demonstrate that (\ref{GH-diam}) fails in general and that geodesics need not converge to geodesics.   Nevertheless there are versions of these
theorems which do hold.

\vspace{.4cm}
\subsection{Review of Ambrosio-Kirchheim Currents on Metric Spaces}

In order to rigorously review the definition of the intrinsic flat distance, one needs a few key results of Ambrosio-Kirchheim.   These results will also be applied later to prove the main theorems of the paper.

In \cite{AK}, Ambrosio-Kirchheim extend Federer-Fleming's notion of integral
currents on Euclidean space to an arbitrary complete metric space, $Z$.   In Federer-Fleming, currents were defined as linear functionals on differential
forms \cite{FF}. This approach extends naturally
to smooth manifolds but not to complete metric spaces which do not have differential
forms.  In the place of differential forms, Ambrosio-Kirchheim use
DiGeorgi's
$m+1$ tuples, $\omega\in \mathcal{D}^m(Z)$,
\be
\omega=f\pi=\left(f,\pi_1 ...\pi_m\right) \in \mathcal{D}^m(Z)
\ee 
where
$f: X \to \R$ is a bounded Lipschitz function and
$\pi_i: X \to \R$ are Lipschitz.  

In \cite{AK} Definitions 2.1, 2.2, 2.6 and 3.1, an $m$ dimensional
current $T\in \curr_m(Z)$ is defined.  Here these are combined into a single
definition:

\begin{defn}
On a complete metric space, $Z$,  an
$m$ dimensional {\bf \em current}, denoted
$T\in \curr_m(Z)$, is a real valued
{\em multilinear functional} on $\mathcal{D}^m(Z)$, with the
following three required properties:
\vspace{.2cm}

\noindent
i) {\em Locality}:
$$ \label{def-locality}
T(f, \pi_1,...\pi_m)=0 \textrm{ if }\exists i\in \{1,...m\} \textrm{ s.t. }\pi_i
\textrm{ is constant on a nbd of } \{f\neq0\}.
$$
ii) {\em Continuity}:
$$
\textrm{Continuity of $T$ with respect to the ptwise convergence
of $\pi_i$ such that $\Lip(\pi_i)\le 1$.}
$$
iii) {\em Finite mass}: 
$$
\exists \textrm{ finite Borel } \mu \,\,\,s.t. \,\,\,
|T(f,\pi_1,...\pi_m)| \le \prod_{i=1}^m \Lip(\pi_i)  \int_Z |f| \,d\mu \,\,\, \forall (f,\pi_1,...\pi_m)\in \mathcal{D}^m(Z).
$$
\end{defn}

%formula for mass will be given in  Lemma~\ref{lemma-weight} when we restrict
%ourselves to integer rectifiable currents.

In \cite{AK} Definition 2.6 Ambrosio-Kirchheim introduce
their mass measure:

\begin{defn}  \label{defn-mass}  
The mass measure $\|T\| $
of a current $T\in \curr_m(Z)$, is the smallest Borel measure $\mu$ such that
\be \label{def-measure-T}
  \Big|T\left(f,\pi\right)\Big|    \le     \int_X |f| d\mu
  \qquad  \forall  \, \left(f,\pi\right) \textrm{ where } \Lip\left(\pi_i\right)\le 1.
\ee  
% Note:
% ||T||\left(A\right)= \sup \{ \sum_j |T\left(f_j,\pi_j\right)| : \sum |f_j| \le \One_A ptwise f_j msbl, \pi_j 1 Lip   AK   Prop 2.7
%
The mass of $T$ is defined
\be \label{def-mass-from-current}
M\left(T\right) = || T || \left(Z\right) = \int_Z \, d\| T\|.
\ee
\end{defn}

In particular
\be \label{eqn-mass}
\Big| T(f,\pi_1,...\pi_m) \Big| \le \mass(T) |f|_\infty \Lip(\pi_1) \cdots \Lip(\pi_m).
\ee

%Stronger versions of locality and continuity, as well as product and
%chain rules are  proven in \cite{AK}[Theorem 3.5].  In particular they
%define $T(f, \pi_1,..., \pi_m)$ for $f$ which are only Borel functions
%as limits of $T(f_j, \pi_1,...,\pi_m)$ where $f_j$ are bounded Lipschitz
%functions converging to $f$ in $L^1(E, ||T||)$.  They also prove
 %\be
%T(f, \pi_\sigma(1),...\pi_\sigma(m))= \sgn(\sigma) T(f, \pi_1,...\pi_m)
%\ee
%for any permutation, $\sigma$, of $\{1,2,...m\}$.

Ambrosio-Kirchheim then define restrictions and push forwards:

\begin{defn} \cite{AK}[Defn 2.5]   \label{defn-rstr}
The {\em restriction} $T\rstr \omega\in \curr_m(Z)$
of a current $T\in M_{m+k}(Z)$ by a $k+1$ tuple
 $\omega=(g,\tau_1,...\tau_k)\in \mathcal{D}^k(Z)$:
\be
(T\rstr\omega)(f,\pi_1,...\pi_m):=T(f\cdot g, \tau_1,...\tau_k, \pi_1,...\pi_m).
\ee
Given a Borel set, $A$, 
\be
T\rstr A := T\rstr \omega
\ee
where $\omega= \One_A \in \mathcal{D}^0(Z)$ is the indicator
function of the set.  In this case,
\be
\mass(T\rstr \omega) = ||T||(A).
\ee
\end{defn}

\begin{defn}\label{defn-push}
Given a Lipschitz map $\varphi:Z\to Z'$, the {\em push
forward} of a current $T\in \curr_m(Z)$
to a current $\varphi_\# T \in \curr_m(Z')$ is given in \cite{AK}[Defn 2.4] by
\be \label{def-push-forward}
\varphi_\#T(f,\pi_1,...\pi_m):=T(f\circ \varphi, \pi_1\circ\varphi,...\pi_m\circ\varphi).
\ee
\end{defn}

\begin{rmrk} \label{rstr-push} 
Observe that
\be
(\varphi_\#T) \rstr (f, \pi_1,...\pi_k))= \varphi_\#(T \rstr (f\circ \varphi, \pi_1\circ\varphi,...\pi_k\circ\varphi) )
\ee
and
\be
(\varphi_\#T )\rstr A = (\varphi_\#T) \rstr (\One_A)
=\varphi_\# (T \rstr (\One_A\circ \varphi)) = \varphi_\#(T \rstr \varphi^{-1}(A)).
\ee
\end{rmrk}

In (2.4) \cite{AK}, Ambrosio-Kirchheim show that
\be  \label{mass-push}
||\varphi_\#T|| \le [\Lip(\varphi)]^m \varphi_\# ||T||,
\ee
so that when $\varphi$ is an isometric
embedding 
\be \label{lem-push-mass}
||\varphi_\#T||=\varphi_\#||T|| \textrm{ and }
\mass(T)=\mass(\varphi_\#T).
\ee

The simplest example of a current is:

\begin{example}\label{basic-current-pushed}
If one has a bi-Lipschitz map, $\varphi:\R^m \to Z$, and 
a Lebesgue function $h\in L^1(A,\Z)$ where $A\subset\R^m$ is Borel,
then $\varphi_\# \lbrack h \rbrack \in \curr_m(Z)$ an 
$m$ dimensional current in $Z$.  Note that
\be
\varphi_\# \lbrack h \rbrack (f,\pi_1,...\pi_m)=
\int_{A \subset \R^m} (h\circ \varphi )(f\circ\varphi) \, 
d(\pi_1\circ \varphi) \wedge \dots \wedge d(\pi_m\circ\varphi)
\ee
where $d(\pi_i\circ\varphi)$ is well defined almost everywhere
by Rademacher's Theorem.  Here the mass measure is
\be
||\lbrack h \rbrack ||= h \,d\mathcal{L}_m
\ee
and the mass is
\be
\mass(\lbrack h \rbrack ) =\int_A h \,d\mathcal{L}_m.
\ee
\end{example}

In \cite{AK}[Theorem 4.6] Ambrosio-Kirchheim define the following set associated with any
integer rectifiable current:   

\begin{defn} \label{defn-set}
The (canonical) set of a current, $T$,
 is the collection of points in $Z$ with positive lower density:
\be \label{def-set-current}
\set\left(T\right)= \{p \in Z: \Theta_{*m}\left( \|T\|, p\right) >0\},
\ee
where the definition of lower density is
\be \label{eqn-lower-density}
\Theta_{*m}\left( \mu, p\right) =\liminf_{r\to 0} \frac{\mu(B_p(r))}{\omega_m r^m}.
\ee
\end{defn}

In \cite{AK} Definition 4.2 and Theorems 4.5-4.6, an integer rectifiable current is defined using the
Hausdorff measure, $\mathcal{H}^m$:

\begin{defn}\label{int-rect-curr}
Let $m\ge 1$.   A current, $T\in \mathcal{D}_m(Z)$, is rectifiable if $\set(T)$ is countably
$\mathcal{H}^m$ rectifiable and if $||T||(A)=0$ for any set
$A\subset Z$ whose Hausdorff measure is zero, $\mathcal{H}^k(A)=0$.   One writes $T\in \rectcurr_m(Z)$.

One says $T\in \rectcurr_m(Z)$ is integer rectifiable,
denoted $T\in \intrectcurr_m(Z)$, if for any $\varphi\in \Lip(Z, \R^m)$ and any open set $A\subset Z$, then 
\be
\exists\, \theta \in \mathcal{L}^1(\R^k,Z) \,\,\,s.t.\,\,\,
\varphi_{\#}(T\rstr A)=\lbrack \theta \rbrack.
\ee   
In fact, $T \in \intcurr_m(Z)$
iff it has a parametrization.  A parametrization $\left(\{\varphi_i\}, \{\theta_i\}\right)$ of an integer rectifiable current $T\in \intrectcurr^m\left(Z\right)$ is a collection of
bi-Lipschitz maps $\varphi_i:A_i \to Z$ with $A_i\subset\R^m$ precompact
Borel measurable and with pairwise disjoint images and
weight functions $\theta_i\in L^1\left(A_i,\N\right)$ such that
\be\label{param-representation}
T = \sum_{i=1}^\infty \varphi_{i\#} \Lbrack \theta_i \Rbrack \quad\text{and}\quad \mass\left(T\right) = \sum_{i=1}^\infty \mass\left(\varphi_{i\#}\Lbrack \theta_i \Rbrack\right).
\ee
A $0$ dimensional rectifiable current is defined by the existence of
countably many distinct points, $\{x_i\}\in Z$,  weights $\theta_i \in \R^+$
and orientation, $\sigma_i \in \{-1,+1\}$
such that 
\be \label{0-param-representation}
T(f)=\sum_h \sigma_i \theta_i f(x_i) \qquad \forall f \in \mathcal{B}^\infty(Z).
\ee
where $\mathcal{B}^\infty(Z)$ is the class of bounded Borel
functions on $Z$ and where
\be
\mass(T)=\sum_h \theta_i< \infty
\ee
If $T$ is integer rectifiable $\theta_i \in \Z^+$, so the sum must be finite.
\end{defn}

In particular, the mass measure of $T \in \intcurr_m(Z)$ satisfies
\be
||T|| = \sum_{i=1}^\infty ||\varphi_{i\#}\Lbrack \theta_i \Rbrack ||.
\ee
Theorems 4.3 and 8.8 of \cite{AK} provide necessary and sufficient
criteria for determining when a current is integer rectifiable.

%The weight function $\theta_T$ is defined almost everywhere on $Z$,
%such that $x\notin \set(T)$, has $\theta_T(x)=0$, and 
%$x=\varphi_i(a)$, has $\theta_T(x)=\theta_i(a)$.

Note that the current in Example~\ref{basic-current-pushed} is an integer
rectifiable current.  

\begin{ex} \label{basic-mani}
If one has a Riemannian manifold, $M^m$,
and a bi-Lipschitz map $\varphi:M^m\to Z$, then 
$T=\varphi_\#\lbrack\One_M\rbrack$
is an integer rectifiable current of dimension $m$ in $Z$.  If $\varphi$
is an isometric embedding, and $Z=M$
then $\mass(T)=\vol(M^m)$.   Note further that
$\set(T)=\varphi(M)$.   
\end{ex}

%\CS{I removed all the long discussion about weight
%as this is only necessary if I try to prove some properties
%of the weight}

\begin{defn}\label{rmrk-def-boundary} \cite{AK}[Defn 2.3] 
The {\em boundary} of $T\in \curr_m(Z)$ is defined
\be \label{def-boundary}
\partial T(f, \pi_1, ... \pi_{m-1}):= T(1, f, \pi_1,...\pi_{m-1}) \in M_{m-1}(Z)
\ee
When $m=0$, set $\partial T=0$.
\end{defn} 

Note that $\varphi_\#(\partial T)=\partial(\varphi_\#T)$. 
  
\begin{defn}  \cite{AK}[Defn 3.4 and 4.2]
An integer rectifiable current  $ T\in\intrectcurr_m(Z)$  is called an 
integral current, denoted $T\in \intcurr_m(Z)$,  if $\partial T$ 
defined as
\be
\partial T \left(f, \pi_1,...\pi_{m-1}\right) := T \left(1, f, \pi_1,...\pi_{m-1}\right)
\ee
has finite mass.   The total mass of an integral current is
\be
\nmass(T)=\mass(T) +\mass(\partial T).
\ee  
\end{defn}

Observe that $\partial \partial T=0$.
In \cite{AK} Theorem 8.6, Ambrosio-Kirchheim prove that 
\be
\partial: \intcurr_m(Z) \to \intcurr_{m-1}(Z)
\ee whenever $m\ge 1$.

Recall Definition~\ref{defn-push} of the push forward
of a current.
By (\ref{mass-push}) one can see that if 
$\varphi: Z_1 \to Z_2$ is Lipschitz, then 
\be
\varphi_{\#}: \intcurr_m(Z_1) \to \intcurr_{m}(Z_2).
\ee

Recall Definition~\ref{defn-rstr} of the restriction of a current.
The restriction of an integral current need not be 
an integral current except in special circumstances.
For example, $T$ might be integration over $[0,1]^2$ with
the Euclidean metric and $A\subset [0,1]^2$ could have
an infinitely long boundary, so that $T\rstr A\notin \intcurr_2([0,1]^2)$
because $\partial(T\rstr A)$ has infinite mass.    The Ambrosio-Kirchheim Slicing Theorem, presented next, allows one to prove
$T\rstr A$ is an integral current for a large collection of open
sets defined using Lipschitz functions.   See in particular 
(\ref{rstr-ok}) below.

\subsection{Ambrosio-Kirchheim Slicing Theorem}\label{subsect-Slicing}

As in the work of Federer-Fleming, Ambrosio-Kirchheim consider the slices of
currents:

\begin{thm} {\bf [Ambrosio-Kirchheim] }\cite{AK}[Theorems 5.6-5.7]
  \label{theorem-slicing}
Let $Z$ be a complete metric space, $T\in \intcurr_m Z$ and $f: Z \to \R$ a Lipschitz function.
For almost every $s\in \R$ one can define an integral current
\be
%BEFORE: <T,f,s> :=  \partial\left( T \rstr f^{-1}\left(-\infty,s]\right) \right) - \left(\partial T\right) \rstr f^{-1}\left(-\infty, s]\right),
<T,f,s> := - \partial\left( T \rstr f^{-1}\left(s,\infty)\right) \right) + \left(\partial T\right) \rstr f^{-1}\left(s, \infty) \right),
\ee
so that
\be \label{bndry-slice-1}
\partial<T,f,s>  = <-\partial T, f,s>
\ee
and $<T_1+T_2, f,s>=<T_1, f,s> + <T_2, f,s>$.
In addition, one can
integrate the masses to obtain:
\be
\int_{s\in\R} \mass(<T,f,s>) \, ds = \mass(T \rstr df)
\le \Lip(f)\, \mass(T)
\ee
where
\be
(T \rstr df)(h, \pi_1,...,\pi_{m-1})=T(h, f,\pi_1,...\pi_{m-1}).
\ee
In particular, for almost every $s>0$ one has
\be \label{rstr-ok}
T \rstr f^{-1}(s,\infty) \in \intcurr_{m-1}\left(Z\right).
\ee
%Furthermore for all Borel sets $A$ we have
%\be
%<T\rstr A, f,s>=<T,f,S> \rstr A
%\ee
%and
%\be
%\int_{s\in\R} ||<T,f,s>||(A) \, ds = ||T \rstr df ||(A).
%\ee
\end{thm}

\begin{rmrk} \label{push-slice-1}
Observe that for any $T\in \intcurr_m(Z')$, and
any Lipschitz functions, $\varphi: Z\to Z'$ and 
$f: Z' \to \R$ and any $s>0$, one has
\be
<\varphi_{\#}T,f,s> = \varphi_{\#} <T , (f \circ\varphi), s >.
\ee
 \end{rmrk}
 
 \subsection{Review of Convergence of Currents}
 
Ambrosio Kirchheim's Compactness Theorem, which extends Federer-Fleming's Flat Norm 
Compactness Theorem, is stated in terms of weak convergence of
currents.   Definition 3.6 of \cite{AK} extends Federer-Fleming's notion of weak convergence (except that they do not require compact support):

\begin{defn} \label{def-weak}
A sequence of integral currents $T_j \in \intcurr_m\left(Z\right)$ is said to converge weakly to
a current $T$ iff the pointwise limits satisfy
\be
\lim_{j\to \infty}  T_j\left(f, \pi_1,...\pi_m\right) = T\left(f, \pi_1,...\pi_m\right) 
\ee
for all bounded Lipschitz $f: Z \to \R$ and Lipschitz $\pi_i: Z \to \R$.
One writes
\be
T_j \weaklyto T
\ee
\end{defn}

One sees immediately that $T_j \weaklyto T$ implies
\be
\partial T_j \weaklyto \partial T,   
\ee
and
\be
\varphi_\# T_j \weaklyto \varphi_\# T.
\ee
%and
%\be
%T_j \rstr (f, \pi_1,..., \pi_k) \weaklyto T\rstr (f, \pi_1,..., \pi_k).   
%\ee
However $T_j \rstr A$ need not converge weakly to $T \rstr A$
as seen in the following example:

\begin{example}
Let $Z= \R^2$ with the Euclidean metric.  Let $\varphi_j: [0,1]\to Z$
be $\varphi_j(t) = (1/j, t)$ and $\varphi_\infty(t) = (0, t)$.   Let 
$S\in \intcurr_1([0,1])$ be
\be
S(f, \pi_1)= \int_0^1 f \, d\pi_1.
\ee
Let $T_j \in \intcurr_1(Z)$
be defined $T_j =\varphi_{j\#}(S)$.   Then $T_j \weaklyto T_\infty$.
Taking $A=[0,1]\times (0,1)$, one has
$T_j \rstr A=T_j$ but $T_\infty \rstr A = 0$.
\end{example}

Immediately below the definition of weak convergence \cite{AK} Defn 3.6,
Ambrosio-Kirchheim prove 
the lower semicontinuity of mass:   {\em
If $T_j$ converges weakly to $T$, then} 
\be
\liminf_{j\to\infty} \mass(T_j) \ge \mass(T).  
\ee
and for any open set, $A\subset Z$,
\be
\liminf_{j\to\infty} ||T_j||(A) \ge ||T||(A).  
\ee

\begin{thm}[Ambrosio-Kirchheim Compactness]\label{AK-compact}
Given any complete metric space 
$Z$, a compact set $K \subset Z$ and $A_0, V_0>0$.
Given
any sequence of integral currents  $T_j \in \intcurr_m \left(Z\right)$ satisfying
\be
\mass(T_j) \le V_0 \textrm{, } \mass(\partial T_j) \le A_0
\textrm{ and }
\set\left(T_j\right) \subset K,
\ee there exists a subsequence, $T_{j_i}$, and a limit current $T \in \intcurr_m\left(Z\right)$
such that $T_{j_i}$ converges weakly to $T$.
\end{thm}

\vspace{.4cm}
\subsection{Review of Integral Current Spaces}

The notion of an integral current space was introduced by
the author and Stefan Wenger in \cite{SorWen2}:

\begin{defn} \label{defn-int-curr-space}
An $m$ dimensional metric space 
$\left(X,d,T\right)$ is called an integral current space if
it has a integral current structure $T \in \intcurr_m\left(\bar{X}\right)$
where $\bar{X}$ is the metric completion of $X$
and $\set(T)=X$.   Also included in the $m$ dimensional integral
current spaces is the $0$ space, denoted $\bf{0}$.   The integral current structure of the $0$ space is $T=0$ and it has an empty metric space.
\end{defn}

Note that $\set\left(\partial T\right) \subset \bar{X}$.   
The boundary of $\left(X,d,T\right)$ is then the integral current space:
\be
\partial \left(X,d_X,T\right) := \left(\set\left(\partial T\right), d_{\bar{X}}, \partial T\right).
\ee
If $\partial T=0$ then one says $\left(X,d,T\right)$ is an integral current without boundary.  The $0$ space has no boundary.

\begin{defn} \label{defn-integral-current-space}
The space of $m\ge 0$ dimensional integral current spaces, 
$\mathcal{M}^m$,
consists of all metric spaces which are integral current spaces 
with currents of dimension $m$ as
in Definition~\ref{defn-int-curr-space} as well as the $\bf{0}$ spaces.
Then $\partial: \mathcal{M}^{m+1}\to \mathcal{M}^{m}$.
\end{defn}

\begin{rmrk} \label{space-param}
Any $m$ dimensional integral current space is countably 
$\mathcal{H}^m$ rectifiable with orientated charts, $\varphi_i$
and weights $\theta_i$ provided
as in (\ref{param-representation}).   
%A $0$ dimensional integral current space
%is a finite collection of points with orientations, $\sigma_i$ and
%weights $\theta_i$ provided as in (\ref{0-param-representation}).
%If this space is the boundary
%of a $1$ dimensional integral current space, then
%as in Remark~\ref{bndry-1-current}, the sum of the signed weights is 0.
\end{rmrk}

\begin{example}
A compact oriented Riemannian manifold with boundary, $M^m$,
is an integral current space, where $X=M^m$, $d$ is the standard
metric on $M$ and $T$ is integration over $M$.  In this
case $\mass(M)=\vol(M)$ and $\partial M$ is the boundary manifold.
When $M$ has no boundary, $\partial M=0$.
\end{example}

\vspace{.4cm}
\subsection{Review of the Intrinsic Flat Convergence}

Recall that the flat distance between $m$ dimensional integral currents 
$S,T\in\intcurr_m\left(Z\right)$ is given by 
\begin{equation} \label{eqn-Federer-Flat}
d^Z_{F}\left(S,T\right):= 
\inf\{\mass\left(U\right)+\mass\left(V\right):
S-T=U+\bdry V \}
\end{equation}
where $U\in\intcurr_m\left(Z\right)$ and $V\in\intcurr_{m+1}\left(Z\right)$.
This notion of a flat distance was first introduced by Whitney
in \cite{Whitney} for chains and later adapted to rectifiable currents by Federer-Fleming \cite{FF}.
The flat distance between Ambrosio-Kirchheim's integral currents was
studied by Wenger in \cite{Wenger-flat}.   In particular,
Wenger proved that if $T_j \in \intcurr_m(Z)$ has
$\mass(T_j) \le V_0$ and $\mass(\partial T_j) \le A_0$ then
$T_j$ converges weakly to $T$ as currents iff 
$d^Z_F(T_j, T) \to 0$
exactly as in Federer-Fleming.

The intrinsic flat distance between integral current spaces
was first defined in \cite{SorWen2}[Defn 1.1]:

\begin{defn} \label{def-flat1} 
 For $M_1=\left(X_1,d_1,T_1\right)$ and $M_2=\left(X_2,d_2,T_2\right)\in\mathcal M^m$ let the 
intrinsic flat distance  be defined:
 \begin{equation}\label{equation:def-abstract-flat-distance}
  d_{\Fm}\left(M_1,M_2\right):=
 \inf d_F^Z
\left(\varphi_{1\#} T_1, \varphi_{2\#} T_2 \right),
 \end{equation}
where the infimum is taken over all complete metric spaces 
$\left(Z,d\right)$ and isometric embeddings 
$\varphi_1 : \left(\bar{X}_1,d_1\right)\to \left(Z,d\right)$ and $\varphi_2: \left(\bar{X}_2,d_2\right)\to \left(Z,d\right)$
and the flat norm $d_F^Z$ is taken in $Z$.
Here $\bar{X}_i$ denotes the metric completion of $X_i$ and $d_i$ is the extension
of $d_i$ on $\bar{X}_i$ and $\phi_\# T$ denotes the push forward of $T$ by the map $\phi$.
\end{defn}

In \cite{SorWen2}, it is observed that
 \be
d_{\Fm}\left(M_1,M_2\right) \le d_{\Fm}\left(M_1,0\right)+d_{\Fm}\left(0,M_2\right) \le \mass\left(M_1\right)+\mass\left(M_2\right).
\ee 
There it is also proven that $d_{\mathcal{F}}$ satisfies the
triangle inequality \cite{SorWen2}[Thm 3.2] and is a distance \cite{SorWen2}[Thm3.27] on the class of precompact integral current spaces up to
current preserving isometry.   In particular it is a distance on the class or oriented compact manifolds with boundary of a given dimension.

In \cite{SorWen2} Theorem 3.23 it is also proven that

\begin{thm}\label{inf-dist-attained}\label{achieved} \cite{SorWen2}[Thm 4.23]
Given a pair of precompact integral current spaces, $M^m_1=(X_1,d_1,T_1)$
and 
$M^m_2=(X_2,d_2,T_2)$, 
there exists a compact metric space, $(Z, d_Z)$,
integral
currents $U\in\intcurr_m\left(Z\right)$ and  $V\in\intcurr_{m+1}\left(Z\right)$,
and isometric embeddings
$\varphi_1 : \bar{X}_1\to Z$ and $\varphi_2:\bar{X}_2 \to Z$
with
\begin{equation} \label{eqn-Federer-Flat-3}
\varphi_{1\#} T_1- \varphi_{2\#} T_2=U+\bdry V
\end{equation}
such that
\begin{equation}\label{eqn-local-defn-2}
d_{\Fm}\left(M_1,M_2\right)=\mass\left(U\right)+\mass\left(V\right).
\end{equation}
\end{thm}

\begin{rmrk}\label{rmrk-inf-dist-attained}
The metric space $Z$ in Theorem~\ref{inf-dist-attained} has
\be
\diam(Z) \le 3 \diam(X_1) + 3\diam(X_2).
\ee
This is seen by consulting the proof of Theorem 3.23 in \cite{SorWen2},
where $Z$ is constructed as the injective envelope of the Gromov-Hausdorff
limit of a sequence of spaces $Z_n$ with this same diameter bound.
\end{rmrk}

The following theorem in \cite{SorWen2} is an immediate consequence
of Gromov and Ambrosio-Kirchheim's Compactness Theorems:

\begin{thm} \label{GH-to-flat}
Given a sequence of precompact $m$ dimensional integral current spaces $M_j=\left(X_j, d_j, T_j\right)$ such that 
\be
\left(\bar{X}_{j}, d_{j}\right) \GHto \left(Y,d_Y\right),\,\,\,
\mass(M_j)\le V_0 \,\,\,\textrm{ and } \,\,\, \mass(\partial M_j)\le A_0
\ee
then a subsequence converges 
 in the 
intrinsic flat sense 
\be
\left(X_{j_i}, d_{j_i}, T_{j_i}\right) \Fto \left(X,d_X,T\right)
\ee
where either $\left(X,d_X,T\right)$ is the ${\bf 0}$ current space
or $\left(X,d_X,T\right)$ is an $m$ dimensional integral current space
with $X \subset Y$ with the restricted metric $d_X=d_Y$.
\end{thm}

Immediately one notes that if $Y$ has Hausdorff dimension less than $m$,
then $(X,d,T)=\bf{0}$.   There are many examples of sequences of Riemannian manifolds which have no Gromov-Hausdorff limit but have an intrinsic flat limit.
The first is Ilmanen's Example of an increasingly hairy three sphere
with positive scalar curvature described in \cite{SorWen2} Example A.7.   

The following three theorems are proven in work of the
author with Wenger \cite{SorWen2}.   These theorems with
the work of Ambrosio-Kirchheim reviewed are key ingredients in the proofs
of the theorems in this paper.

\begin{thm}\label{converge}\label{converge} \cite{SorWen2}[Thm 4.2]
If a sequence of %precompact 
integral current spaces has
\be
M_{j}=\left(X_j, d_j, T_j\right) \Fto M_0=\left(X_0,d_0,T_0\right),
\ee 
then
there is a separable
complete metric space, $Z$, and isometric embeddings  $\varphi_j: X_j \to Z$ such that
\be
d_F^Z(\varphi_{j\#}T_j,\varphi_{0\#} T_0)\to 0
\ee
and thus $\varphi_{j\#}T_j$ converges weakly to $\varphi_{0\#} T_0$ as well. 
\end{thm}

\begin{thm}\label{convergeto0} \cite{SorWen2}[Thm 4.3]
If a sequence of 
integral current spaces has
\be
M_{j}=\left(X_j, d_j, T_j\right) \Fto {\bf{0}}
\ee 
then
one may choose points $x_j\in X_j$ and a
separable
complete metric space, $Z$, and isometric embeddings  $\varphi_j: X_j \to Z$ such that
$\varphi_j(x_j)=z_0\in Z$ and
\be
d_F^Z(\varphi_{j\#}T_j, 0)\to 0
\ee
and thus
$\varphi_{j\#}T_j$ converges weakly to $0$ in $Z$ as well.
\end{thm}

Theorems~\ref{converge} and~\ref{convergeto0}
 combined with Ambrosio-Kirchheim's lower semicontinuity
 of mass [c.f. Remark~\ref{semi-mass}] imply the following:

\begin{thm}\label{semi-mass}
If a sequence of integral current spaces
$M_{j}$ converges in the intrinsic flat sense
to an %note not required to be precompact 
integral current space,
 $M_\infty$, then
 \be
 \liminf_{i\to\infty} \mass(M_i) \ge \mass(M_\infty)
 \ee
 \end{thm}
 
 Note that Theorems~\ref{converge}, ~\ref{convergeto0} and ~\ref{semi-mass} do not
 require uniform bounds on the masses or volumes of the $M_j$ and $\partial M_j$.

\subsection{Balls in Integral Current Spaces}\label{subsect-Iballs}

Many theorems in Riemannian geometry involve open and closed balls,
\be
B(p,r)= \{ x \in X : \, d_X(x,p)<r\} \quad 
\bar{B}(p,r)= \{ x \in X : \, d_X(x,p)\le r\}.  
\ee  
Here a few basic lemmas are proven about balls in integral current
spaces.   These lemmas are new but so basic that they are
best placed in this background section.

\begin{lem}\label{lem-ball}
A ball in an integral current space, $M=\left(X,d,T\right)$,
with the current restricted from the current structure of the Riemannian manifold is an integral current space itself, 
\be\label{Spr}
S(p,r):=\left(\,\set(T\rstr B(p,r)),d,T\rstr B\left(p,r\right)_{\textcolor{white}{.}}\right)
\ee
for almost every $r > 0$.   Furthermore,
\be\label{ball-in-ball}
B(p,r) \subset \set(S(p,r))\subset \bar{B}(p,r)\subset X.
\ee
\end{lem}

\begin{proof}
First one shows that $S(p,r)=T\rstr B(p,r)$ is an integer rectifiable current.   
Let $\rho_p: \bar{X} \to \R$ be the distance function from $p$.  Then
by Ambrosio-Kirchheim's Slicing Theorem, applied to
$f(x)=-\rho_p(x)$, one has
\begin{eqnarray}
\partial ( T \rstr B(p,r) ) &=& \partial ( T \rstr \rho_p^{-1}(-\infty, r) )\\
& = & <T, -\rho_p, -r> + (\partial T) \rstr \rho_p^{-1}(-\infty,r) \\
& = & <T, -\rho_p, -r> + (\partial T) \rstr B(p,r)
\end{eqnarray}
where the mass of the slice $<T,\rho_p,r>$ is bounded for almost every $r$.
Thus
\begin{eqnarray}
\,\,\,\,\mass(\partial ( T \rstr B(p,r) ))
& \le & \mass(<T, -\rho_p, -r>) + \mass((\partial T) \rstr B(p,r)) \\
& \le & \mass(<T, -\rho_p, -r>) \,\,\,\,+ \,\,\,\,\,\mass(\partial T)\,\,\,\,\,\,<\,\,\,\,\, \infty.
\end{eqnarray}
So $S(p,r)$ is an integral current in $\bar{X}$ for almost every $r$.

Next one proves (\ref{ball-in-ball}).
Recall that $x\in \set(S(p,r)) \subset \bar{X}$ iff
\begin{eqnarray} \label{ball-in-ball-1}
0&<&\liminf_{s\to 0} \frac{||S(p,r)||(B(x,s))}{\omega_m s^m} \\
&=&\liminf_{s\to 0} \frac{||T||(B(p,r)\cap B(x,s))}{\omega_m s^m}
\end{eqnarray}
If $x\in B(p,r) \subset X$, then 
eventually $B(x,s)\subset B(p,r)$ and the liminf
is just the lower density of $T$ at $x$.  Since $x\in X=\set(T)$,
this lower density is positive.
If $x\in \bar{X}\setminus X$, then the liminf is $0$ because 
it is smaller than the density of $T$ at $x$, which is $0$.
If $x\notin \bar{B}(p,r)$, then the liminf is 0 because
eventually the balls do not intersect.   
\end{proof}

One may imagine that it is possible that a ball is cusp shaped
and that some points in the
closure of the ball that lie in $X$ do not lie in the set of $S(p,r)$.
In a manifold, the set of $S(p,r)$ is a closed ball:

\begin{lem}
When $M$ is a Riemannian manifold with boundary we have
\be
S\left(p,r\right)=\left(\bar{B}\left(p,r\right),d,T\rstr B\left(p,r\right)\right)
\ee
is an integral current space for all $r > 0$.
\end{lem}

\begin{proof}
In this case,
\begin{eqnarray}
\partial (T \rstr B(p,r)) (f, \pi_1,...,\pi_m)
&=& (T \rstr B(p,r) ) (1,f, \pi_1,...,\pi_m)\\
&=& T (\chi_{B(p,r)},f, \pi_1,...,\pi_m)\\
&=& \int_M \chi_{B(p,r)} df\wedge d\pi_1\wedge \cdots \wedge d\pi_m\\
&=& \int_{B(p,r)} df\wedge d\pi_1\wedge \cdots \wedge d\pi_m\\
&=& \int_{B(p,r)} df\wedge d\pi_1\wedge \cdots \wedge d\pi_m\\
&=& \int_{\partial B(p,r)} f \, d\pi_1\wedge \cdots \wedge d\pi_m
\end{eqnarray}
So $\mass(\partial (T\rstr B(p,r))) = \vol_{m-1}(\partial B_p(r))<\infty$.

Observe that $\bar{B}(p,r)\subset M$ is $\set(S(p,r))$, 
by (\ref{ball-in-ball-1}).   If $d(x,p)=r$, then let $\gamma:[0,r]\to M$
be a curve parametrized by arclength running minimally
from $x$ to $p$.  Then
\be
B(\gamma(s/2), s/2) \subset B(x,s)\cap B(p,r).
\ee 
and
\begin{eqnarray} \label{ball-in-ball-2}
\liminf_{s\to 0} \frac{||S(p,r)||(B(x,s))}{\omega_m s^m} 
&=&\liminf_{s\to 0} \frac{||T||(B(p,r)\cap B(x,s))}{\omega_m s^m}\\
&\ge&\liminf_{s\to 0} \frac{||T||(B(\gamma(s/2), s/2)}{\omega_m s^m}\\
&\ge&\liminf_{s\to 0} \frac{\vol(B(\gamma(s/2), s/2)}{2^m\omega_m (s/2)^m}
\ge \frac{1/2}{2^m}
\end{eqnarray}
because in a manifold with boundary, the balls eventually lie within a 
half plane chart where all tiny balls are either uniformly close to a Euclidean
ball or half a Euclidean ball.
\end{proof}

\begin{example}\label{bad-level}
There exist integral current spaces with balls that
are not integral current spaces.
\end{example}

\begin{proof}
Suppose one defines an integral current space, $(X,d,T)$ where
$X=S^2$ with the following generalized metric
\be
g= dr^2 + (\cos(r)/r^2)^2 d\theta^2 \qquad r\in [-\pi/2,\pi/2].
\ee
The metric is defined as
\be
d(p_1, p_2) =\inf\{ L_g(\gamma): \,\, \gamma(0)=p_1, \, \gamma(1)=p_2\}
\ee
where
\be
L_g(\gamma)=\int_0^1 g(\gamma'(t), \gamma'(t))^{1/2} \, dt
\ee
as in a Riemannian manifold.   In fact this metric space consists
of two open isometric Riemannian manifolds diffeomorphic to disks whose
metric completions are glued together along corresponding points.
The current structure $T$ is defined by
\begin{eqnarray}
T(f, \pi_1,...,\pi_m) &=& \int_{-\pi/2}^{\pi/2} \int_{S^1} f \, d\pi_1 \wedge \cdots \wedge d\pi_m  \\
&=& \int_{-\pi/2}^{0} \int_{S^1} f \, d\pi_1 \wedge \cdots \wedge d\pi_m  
\\
&&\qquad +\,\,\, \int_{0}^{\pi/2} \int_{S^1} f \, d\pi_1 \wedge \cdots \wedge d\pi_m  
\end{eqnarray}
so that $\partial T=0$ and
\be
\mass(T)\,=\, \vol_m\left(r^{-1}[-\pi/2,0) \right)
+\vol_m\left(r^{-1}(0,\pi/2] \right) 
\,\,<\,\infty.
\ee

Setting $p$ such that $r(p)=-\pi/2$, then
$S(p, \pi/2)$ is a rectifiable current but its boundary does
not have finite mass.  This can be see by taking $q$ such
that $(r(q),\theta(q))=(0,0)$, setting $\pi_1=\rho_q$ 
and $f=\rho_p=r+\pi/2$ and observing that
\begin{eqnarray}
|\partial(S(p,\pi/2))(f, \pi_1)| &=& |S(p,\pi/2)(1, f, \pi_1)|\\
&=& \left| \int_{B(p,\pi/2)} df \wedge d\pi_1 \right| \\
&\ge& \left| \int_{B(p, \pi/2-\delta)} df \wedge d\pi_1 \right| \\
&=& \left| \int_{\partial B(p,\pi/2-\delta)} f \, d\pi_1 \right| \\
&=& \left| \int_{\theta=-\pi}^{\pi} (\pi/2-\delta) \, \frac{d\pi_1}{d\theta} \, d\theta \right| \\
&=& \left| \int_{\theta=-\pi}^{\pi} (\pi/2-\delta) \, \frac{\cos(r)}{r^2} \, d\theta \right| \\
&\ge &  (\pi/2 -\delta) \,\frac{\,\cos(-\delta)\,}{\delta^2} \,2\pi
\end{eqnarray}
which is unbounded as $\delta$ decreases to $0$.
\end{proof}

\begin{rmrk}\label{outside-balls}
Note that
the outside of the ball, $(M\setminus B(p,r), d, T-S(p,r))$, is
also an integral current space for almost every $r>0$.
\end{rmrk}

\begin{rmrk}\label{rmrk-ball-bound}
In some of the theorems in this paper, it will be important to
estimate $d_{\mathcal{F}}(S(p,r), \bf{0})$.   There are various 
ways to estimate this value.   First observe that
\be
d_{\mathcal{F}}\left(S(p,r), {\bf{0}}\right) 
\le \min \left\{\, \mass(S(p,r)),_{\textcolor{white}{,}} \mass(\partial(S(p,r))\right\}.
\ee
In addition, if one finds a comparison integral current space, $N$, 
such that
\be
d_{\mathcal{F}}(S(p,r), N) < d_{\mathcal{F}}(N,{\bf{0}})/2
\ee
then by the triangle inequality
\be
d_{\mathcal{F}}(S(p,r), {\bf{0}}) > d_{\mathcal{F}}(N,{\bf{0}})/2.
\ee
Recall that in joint work with Wenger \cite{SorWen2}, in
joint work with Lakzian \cite{Lakzian-Sormani}, and in joint
work with Lee \cite{LeeSormani1} various means of
estimating the intrinsic flat distance are provided.  
\end{rmrk}

\section{Converging Points and Diameters}

In this section the limits of points in
sequences of integral current spaces that converge in the
intrinsic flat sense are examined.   See Definitions~\ref{point-conv} and~\ref{point-Cauchy}
and Lemma~\ref{to-a-limit}.  The diameter is then proven to be lower semicontinuous.  See Definition~\ref{defn-diam} and  Theorem~\ref{diam-semi}. \footnote{Some of these notions were original defined in an older version of \cite{Sormani-properties} 
but they are now moved here
and will only be reviewed there.}

Before beginning, recall that Theorem~\ref{converge} which was proven in work of the author
with Wenger in \cite{SorWen2} states that a sequence of manifolds
which converges in the intrinsic flat sense can be isometrically embedded
into a common metric space.    This theorem is applied
to define the notion
of a converging sequence of points:

\begin{defn} \label{point-conv}
If $M_i=(X_i, d_i,T_i) \Fto M_\infty=(X_\infty, d_\infty,T_\infty)$, 
then one says $x_i\in X_i$ are a converging sequence that converge to
$x_\infty\in \bar{X}_\infty$ if there exists a complete metric space
$Z$ and isometric embeddings 
$\varphi_i:X_i\to Z$ such that 
$\varphi_{i\#} T_i \Fto \varphi_{\infty\#}T_\infty$ and 
$\varphi_i(x_i) \to \varphi_\infty(x_\infty)$.   One says 
a collection of
points, $\{p_{1,i}, p_{2,i},...p_{k,i}\}$,
converges to a corresponding collection of points, 
$\{p_{1,\infty}, p_{2,\infty},...p_{k,\infty}\}$, if 
$\varphi_{i}(p_{j,i}) \to \varphi_\infty(p_{j, \infty})$ for $j=1..k$.
\end{defn}

Unlike in Gromov-Hausdorff convergence, there is a possibility of disappearing sequences of points:

\begin{defn} \label{point-Cauchy}
If $M_i=(X_i, d_i,T_i) \Fto M_\infty=(X_\infty, d_\infty,T_\infty)$, then one says $x_i\in X_i$ 
are Cauchy if there exists a complete metric space
$Z$ and isometric embeddings 
$\varphi_i:M_i\to Z$ such that 
$\varphi_{i\#} T_i \Fto \varphi_{\infty\#}T_\infty$ and 
$\varphi_i(x_i) \to z_\infty \in Z$.   One says the
sequence is disappearing if $z_\infty \notin \varphi_\infty(X_\infty)$.
One says the sequence has no limit in $\bar{X}_\infty$ if
$z_\infty \notin \varphi_\infty(\bar{X}_\infty)$.
\end{defn}

\begin{rmrk}\label{rmrk-disappear}
Examples with disappearing splines from \cite{SorWen2}
demonstrate that there exist Cauchy sequences of points
which disappear.  In fact $z_\infty$ may not even lie in the metric
completion of the limit space, $\varphi_\infty(\bar{X}_\infty)$..   
\end{rmrk}

\begin{lem}\label{to-a-limit}
If a sequence of integral current spaces, $M_i=\left(X_i,d_i,T_i\right)\in \mathcal{M}_0^m$, 
converges to 
an integral current space, $M=\left(X,d,T\right)\in \mathcal{M}_0^m$, in the intrinsic flat sense, then every point $x$ in the limit space
$X$ is the limit of points $x_i\in M_i$.  
In fact there exists a sequence of maps $F_i: X \to X_i$
such that $x_i=F_i(x)$ converges to $x$ and
\be
\lim_{i\to \infty} d_i(F_i(x), F_i(y))= d(x,y) \,\,\, \forall x,y\in X.
\ee
\end{lem}

This sequence of maps $F_i$ are not uniquely defined and
are not even unique up to isometry.

\begin{proof}
By Theorem~\ref{converge} there exists a common metric space
$Z$ and isometric embeddings $\varphi_i: X_i \to Z$ and $\varphi:X \to Z$ such that
\be 
\varphi_{\#}T-\varphi_{i\#}T_i=U_i +\partial V_i
\ee
where $m_i=\mass\left(U_i\right)+\mass\left(V_i\right)\to 0$.
So $\varphi_{i\#}T_i$ converges in the flat and the weak sense to $\varphi_\# T$.

Let $\rho_x$ be the distance
function from $\varphi\left(x\right)$.
Since $x\in \spt(T)$, for any $\varepsilon>0$, 
\be
||\varphi_{\#}T||(\rho_x^{-1}[0,\varepsilon))>0.
\ee
By the lower semicontinuity of mass, 
\be
\liminf_{i\to\infty} 
||\varphi_{i\#}T_i||\left(\rho_x^{-1}[0,\varepsilon)\right)
\ge ||\varphi_{\#}T||\left(\rho_x^{-1}[0,\varepsilon)\right)>0.
\ee
In particular,
\be
\exists N_{\epsilon,x} \in \N \,\,s.t.
\,\,\varphi_{i\#}T_i\rstr \left(\rho_x^{-1}[0,\varepsilon)\right)
\neq 0 \qquad \forall i \ge N_{\epsilon,x}.
\ee
  So for all $x\in X$
and any $j\in \N$
\be \label{eqn-sie-here-compact}
\exists N_{j,x}\,s.t.\, \exists s_{i,j,x} \in \set(\varphi_{i\#}T) \cap B\left(x,1/j \right)
\qquad \forall i \ge N_{j,x}.
\ee
Without loss of generality, assume $N_{j,x}$ is increasing in $j$.
For $i\in \{1,..., N_{1,x}\}$ take $j_i=1$.   Then for
$i\in \{N_{j-1,x} +1, ..., N_{j,x}\}$ let $j_i=j$.  Thus $i\ge N_{j_x}$
Let
\be
x_i= \varphi_i^{-1}(s_{i, j_i,x}).
\ee
Then $\varphi_i(x_i) \in B(x, 1/j_i)$ and $\varphi_i(x_i)\to \varphi(x)$.

Since this process can be completed for any $x\in X$, one
has defined maps $F_i: X \to X_i$ such that 
\be
\varphi_i(F_i(x)) \to \varphi(x).
\ee
Finally, for all $x,y\in X$, 
\be
d_i(F_i(x),F_i(y))= d_Z(\varphi_i(F_i(x)), \varphi_i(F_i(y)))
\to d_Z(\varphi(x), \varphi(y)) =d(x,y).
\ee
\end{proof}

\begin{defn}\label{defn-diam}
Like any metric space, one can define the diameter
of an integral current space, $M=(X,d,T)$, to be
\be
\diam(M)=\sup\left\{ d_X(x,y): \,\, x, y\in X\} \in [0,\infty]\right\}.
\ee
In addition, explicitly define the diameter of the
$0$ integral current space to be $0$.
A space is bounded if the diameter is finite.
\end{defn}

\begin{thm} \label{diam-semi}
Suppose $M_i \Fto M$ are integral current spaces then
\be
\diam(M) \le \liminf_{i\to \infty} \diam(M_i) \subset [0,\infty]
\ee
\end{thm}

\begin{proof}
Note that by the definition, $\diam(M_i)\ge 0$, so
the liminf is always $\ge 0$.  Thus the inequality is trivial when
$M$ is the $0$ space.  Assuming $M$ is not the $0$ space, 
for any
$\epsilon>0$, there exists $x,y\in X$ such that 
\be
\diam(M) \le d(x,y) +\epsilon.
\ee
By Lemma~\ref{to-a-limit}, there exists $x_i, y_i\in X_i$
converging to $x,y \in X$ so that
\be
\diam(M) \le \lim_{i\to\infty} d_i(x_i,y_i) +\epsilon
\le \liminf_{i\to\infty} \diam(X_i) +\epsilon. 
\ee
\end{proof}

%%%%%%%%%%%%%%%%%%%%%%
\section{Convergence of Balls and Spheres}

In this section the following key lemma concerning the convergence of balls and spheres is proven.   It is an essential
ingredient when trying to prove intrinsic flat limits are not the zero space
or that points do not disappear.   See Remark~\ref{rmrk-balls-converge}.
It will be applied to prove Theorem~\ref{BW-BASIC}, Theorem~\ref{Arz-Asc-Unif-Local-Isom}, and Example~\ref{ex-no-limit}. 

\begin{lem}\label{balls-converge}
If $M_j =(X_j,d_j,T_j)\Fto M_\infty=(X_\infty,d_\infty,T_\infty)$ and $p_j \to p_\infty\in \bar{X}_\infty$, then there exists a subsequence of $M_j$ also denoted
$M_j$  such that for almost every $r>0$,
\be
S(p_j,r)= \left(\bar{B}\left(p_j,r\right),d_j,T_j\rstr B\left(p_j,r\right)\right)
\ee
are integral current spaces for $j\in \{1,2,...,\infty\}$ and 
\be
S(p_j,r) \Fto S(p_\infty,r).
\ee
If $p_j$ are Cauchy with no limit in $\bar{X}_\infty$
then there exists $\delta>0$ such that
for almost every $r\in (0,\delta)$ 
such that
$S(p_j,r)$
are integral current spaces for $j\in \{1,2,...\}$ and 
\be\label{to-0}
S(p_j,r) \Fto 0.
\ee
If $M_j \Fto \bf{0}$ then for almost every $r$ and for all sequences $p_j$
one has (\ref{to-0}).
\end{lem}

In Example~\ref{ex-balls-converge} demonstrates
why it is necessary
to choose a subsequence.   Observe that this lemma does not require a
uniform upper bound on volume and boundary volume.

\begin{rmrk}\label{rmrk-balls-converge}
The first part
of this lemma was stated as a lemma and
applied by the author and Stefan Wenger to
prove the intrinsic flat and Gromov-Hausdorff limits of noncollapsing
sequences of Riemannian manifolds with nonnegative Ricci curvature
agree in \cite{SorWen2}.  A reference to a proof of a related lemma by Ambrosio-Kirchheim \cite{AK}
was provided there.  As the lemma in \cite{AK} did allow for
changing basepoints $p_j\neq p_\infty$, it was not completely clear
to everyone how one should prove this lemma.   So it is essential
to provide full details here.
\end{rmrk}

Lemma~\ref{balls-converge} is now proven:

\begin{proof}
By Theorem~\ref{converge} and~\ref{convergeto0}
there exists a common complete metric space,
$Z$, and isometric embeddings, $\varphi_j: X_j \to Z$ and $\varphi_\infty:X_\infty \to Z$, such that
\be \label{A-j-B-j-1} 
\varphi_{j\#}T_j-T=\partial B_j + A_j
\ee
where $A_j\in \intcurr_m(Z)$ and $B_j\in \intcurr_{m+1}(Z)$ 
with
\be
\mass(A_j)+\mass(B_j) \to 0
\ee
and where
\be
T=\varphi_{\infty\#}T_\infty\in \intcurr_m(Z) 
\textrm{ when } M_\infty \neq {\bf{0}} \textrm{ and } T=0
\textrm{ when } M_\infty={\bf{0}}.
\ee

Since $p_j$ are Cauchy,  
\be
z_j=\varphi_j(p_j) \to z_\infty\in Z.
\ee
When $p_j\to p_\infty$ then $z_\infty=\varphi_\infty(p_\infty)$.  Then
for almost every $r$
\be
(\varphi_{j\#}T_j)\rstr B(z_j, r)= \varphi_{j\#}S(p_j,r).
\ee
and
\be
T\rstr B(z_\infty, r)= \varphi_{\infty\#}S(p_\infty,r).
\ee
If $p_j$ has no limit in $\bar{X}_\infty$, then 
$z_\infty\notin \varphi_\infty(\bar{X}_\infty)$ and so there exists $\delta>0$
such that for all $r<\delta$,
\be
B(z_\infty, r) \cap \varphi_\infty(\bar{X}_\infty)=0.
\ee
So
\be
T \rstr B(z_\infty, r) =0.
\ee
If $M_j \Fto 0$, then one has this as well without requiring $r<\delta$.

So to prove the theorem in all cases one need only show that for almost every
$r$ one can find a subsequence of the $M_j$ also denoted $M_j$
such that $S(p_j, r)$ are integral
current spaces and 
\be\label{only-need-this}
d_F^Z\left((\varphi_{j\#}T_j)\rstr \rho_j^{-1}(-\infty,r), T \rstr \rho_\infty^{-1}(-\infty,r)\right) \to 0
\ee
where $\rho_j(z)=d_Z(z_j,z)$.

By Lemma~\ref{lem-ball} for almost every $r$ these are
integral current spaces.

Observe that by (\ref{A-j-B-j-1}), for almost every $r$: 
\begin{eqnarray}
\,\,\,\,\,\,(\varphi_{j\#}T_j) \rstr \rho_j^{-1}(-\infty,r)& - &T\rstr \rho_j^{-1}(-\infty,r)\,=\\
&&=\,\,\,(\partial B_j) \rstr \rho_j^{-1}(-\infty,r)\,+ \,\,A_j\rstr  \rho_j^{-1}(-\infty,r)\\
&&= \,<B_j, -\rho_j, -r> + \,\,\partial \left(B_j \rstr \rho_j^{-1}(-\infty,r)\right)\\
&&\qquad \qquad +\,\,\,\,A_j\rstr  \rho_j^{-1}(-\infty,r).
\end{eqnarray}
Thus $
d_F^Z\left(\varphi_{j\#}T_j \rstr \rho_j^{-1}(-\infty,r) ,T\rstr \rho_j^{-1}(-\infty,r)\right)\le$
\begin{eqnarray}
&\le& f_j(r) + \mass(B_j \rstr \rho_j^{-1}(-\infty,r))
+ \mass(A_j\rstr  \rho_j^{-1}(-\infty,r)) \\
&\le& f_j(r) + \mass(B_j) + \mass(A_j)
\end{eqnarray}
where
\be
f_j(r)=\mass(<B_j, -\rho_j, -r>).
\ee  
By the Ambrosio-Kirchheim Slicing Theorem 
\begin{eqnarray}
\int_{-\infty}^{\infty} f_j(r) \, dr &=&\int_{-\infty}^{\infty} \mass(<B_j, \rho_j, r>) \, dr \\
&=& \mass(B_j\rstr d\rho_j) \le \Lip(\rho_j) \mass(B_j)\le \mass(B_j) \to 0.
\end{eqnarray}
Since $f_j$ converge in $L^1$ to $0$, there exists a subsequence, also denoted $f_j$, such that
for almost every $r>0$, $f_j(r)$ converge to $0$ pointwise (c.f. \cite{Rudin-R&C} Theorem 3.12).

Thus there is a subsequence such that for almost every $r>0$
\be\label{ok-here}
\lim_{j\to \infty}
d_F^Z\left(\varphi_{j\#}T_j \rstr \rho_j^{-1}(-\infty,r) ,T\rstr \rho_j^{-1}(-\infty,r)\right)=0.
\ee
Next observe that the set
\be
K=\left(\rho_j^{-1}(-\infty,r)\setminus \rho_\infty^{-1}(-\infty,r)\right)
\cup 
\left(\rho_j^{-1}(-\infty,r)\setminus \rho_\infty^{-1}(-\infty,r)\right)
\ee
satisfies
\be
K \subset \rho_\infty^{-1}(r-\delta_j, r+\delta_j) 
\ee
where 
\be
\delta_j=d_Z(z_j, z_\infty).
\ee
Then 
\begin{eqnarray*}
d_F^Z\left(T\rstr \rho_j^{-1}(-\infty,r), T\rstr \rho_\infty^{-1}(-\infty,r)\right)
&\le & \mass \left(T\rstr \rho_j^{-1}(-\infty,r)\,-\, T\rstr \rho_\infty^{-1}(-\infty,r)\right)\\
&\le & \mass( T \rstr K) \\
&\le & ||T||\left(\rho_\infty^{-1}(r-\delta_j, r+\delta_j) \right)
 \end{eqnarray*}
Since $\lim_{j\to\infty} \delta_j=0$, one has
\begin{eqnarray} 
 \lim_{j\to\infty} ||T||\left(\rho_\infty^{-1}(r-\delta_j, r+\delta_j) \right)
& =&
 \lim_{j\to\infty} ||f_\#T||\left(\rho_\infty^{-1}(r-\delta_j, r+\delta_j) \right)\\
 &=&||f_\#T||\{r\}
 \end{eqnarray}
Since $||f_\#T||$ is a finite measure on $\mathbb{R}$,
$||f_\#T||\{r\}=0$ except on a countable set of values of $r$.
Thus, for almost every $r$,
 \be
 \lim_{j\to \infty}d_F^Z(T\rstr \rho_j^{-1}(-\infty,r), T\rstr \rho_\infty^{-1}(-\infty,r)) =0.
 \ee
 Combining this with (\ref{ok-here}) one has (\ref{only-need-this}) and the
 proof is complete.
\end{proof}

\begin{example}\label{ex-balls-converge}
There exists a sequence of Riemannian manifolds $M_j$
diffeomorphic to a torus with $vol(M_j)\le V_0$ such that
$M_j \Fto 0$ but there exists a Cauchy sequence $p_j\in M_j$ such that
$S(p_j,r)$ does not have an intrinsic flat limit for any $r\in(0,\pi)$. 
\end{example}

\begin{proof}
Take the metric
\be
g_j= dr^2 + f_j^2(r) d\theta^2\qquad r\in[0, \pi]
\ee
with $f_j(0)=0$, $f_j(\pi)=0$, $f_j'(0)=1$, $f_j'(\pi)=-1$
so that $M_j$ is a smooth Riemannian manifold.
Choose $f_j>0$ smooth on $(0,\pi)$ such that
\be
\int_0^{\pi} f^2_j(r) \, dr \to 0
\ee
and such that 
\be
f_j(r)>1\textrm{ for } r\in [j \textrm{ mod }\pi, j+1/j \textrm{ mod }\pi]\cap (1/j^2, \pi-1/j^2)
\ee
and
\be
f_j(r)<1/j\textrm{ for } r\in [j+2/j \textrm{ mod }\pi, j+3/j \textrm{ mod } \pi]\cap (1/j^2, \pi-1/j^2)
\ee
and $f_j$ smoothly decreasing in between.
Since
\be
\vol(M_j)=4\pi\int_0^{2\pi} f^2_j(r) \, dr \to 0
\ee
one has $M_j \Fto \bf{0}$.   Take $p_j$ to be the point where $r=0$.
Suppose one has $r'$ such that the balls converge to the
zero integral current space, $S(p_j,r')\Fto {\bf{0}}$, then
the spheres also converge to the zero space, $\partial S(p_j,r')\Fto {\bf{0}}$.

However there exists a subsequence $j'\to \infty$ such that
$r\in [j'\textrm{ mod }\pi, j+1/j' \textrm{ mod }\pi]$.   On this set $S(p_{j'},r)$ is
bi-Lipschitz close to a circle $S^1$ endowed with the restricted
metric from the disk.   So 
\be
\partial S(p_{j'},r) \Fto \left(S^1, d_{D^2}, \int_{S^1}\right).
\ee
\end{proof}

Also useful for some applications is the following lemma:

\begin{lem} \label{rescaling}
Let $M_j=(X_j, d_j, T_j)$ and let $R>0$.  Then one
has rescaled integral current spaces, $M'_j=(X_j, d_j/R, T_j)$,
one of which may possibly be $\bf{0}$,
and
\be
d_{\mathcal{F}}(M_1, M_2) \le d_{\mathcal{F}}(M'_1, M'_2)R^m (1+ R).  
\ee
In particular taking almost any $r=R\in (0,\delta)$ and $p_j \in X_j$ 
one can rescale
\be
S(p_j,r)=\left({\set}(T_j\rstr B(p_j,r)),d_j,T_j\rstr B\left(p_j,r\right)\right)
\ee
by $r$ to obtain
\be
S'(p_j,1)=\left({\set}(T_j\rstr B(p_j,1)),d_j/R,T_j\rstr B\left(p_j,r\right)\right)
\ee
and
\be
d_{\mathcal{F}}(S(p_1,r), S(p_2,r)) \le 
d_{\mathcal{F}}(S'(p_1,1), S'(p_2,1)) r^m (1+\delta). 
\ee
\end{lem}

\begin{proof}
By the Theorem~\ref{achieved}, there exists isometric embeddings
 $\varphi_j: X_j \to Z$ 
\be
d_{Z}(\varphi_j(x), \varphi_j(y))/R = d_j(x,y)/R \qquad \forall x,y \in X_j
\ee
and $A\in \intcurr_m(Z)$, $B\in \intcurr_{m+1}(Z)$
such that
\be
\varphi_{1\#}T_1 -\varphi_{2\#}T_2 = A +\partial B
\ee
and 
\be
d_{\mathcal{F}}(M'_1, M'_2)=\mass(A) +\mass(B)
\ee
where these masses are defined using $d_Z/R$.
Then $\varphi_j: X_j \to Z$ 
\be
d_{Z}(\varphi_j(x), \varphi_j(y)) = d_j(x,y) \qquad \forall x,y \in X_j
\ee 
and so by definition of intrinsic flat distance
\be
d_{\mathcal{F}}(M_1, M_2)\le \mass'(A) +\mass'(B)
\ee
where these masses are defined using $d_Z$.  Thus
\begin{eqnarray}
d_{\mathcal{F}}(M_1, M_2) &\le&  \mass(A)R^m +\mass(B)R^{m+1}\\
&\le& ( \mass(A) +\mass(B)) R^{m}(1+R)\\
&\le& d_{\mathcal{F}}(M'_1, M'_2)R^m (1+ R).  
\end{eqnarray}
It is easy to see this argument also works when $M_2={\bf{0}}$ taking
$\varphi_{2\#}T_2=0$.   
\end{proof}

%%%%%%%%%%%%%%%%%%%%%%%
\vspace{.4cm}
\section{Flat convergence to Gromov-Hausdorff Convergence}\label{sect-flat-to-GH} \label{sect-flat-to-GH}

In this subsection, Theorem~\ref{flat-to-GH} is proven:

\begin{thm} \label{flat-to-GH}\footnote{This theorem and its proof originally appeared an early preprint 
version of \cite{Sormani-properties} but has now been moved to this paper with minor corrections.  It will not appear in any publication of \cite{Sormani-properties}.}
If a sequence of precompact integral current spaces, $M_i=\left(X_i,d_i,T_i\right)\in \mathcal{M}_0^m$, 
converges to 
a nonzero precompact integral current space, $M=\left(X,d,T\right)\in \mathcal{M}_0^m$, in the intrinsic flat sense, 
then there exists $S_i \in \intcurr_m\left(\bar{X}_i\right)$ such that
$N_i=\left(\set \left(S_i\right), d_i\right)$ converges to $\left(\bar{X},d\right)$ in the Gromov-Hausdorff
sense
\be\label{flat-to-GH-0}
d_{GH}(N_i,M) \to 0
\ee
 and
\be \label{flat-to-GH-1}
\liminf_{i\to\infty}\mass(S_i) \ge \mass(M).
\ee
When the $M_i$ are Riemannian manifolds, the $N_i$ can be taken to be
settled completions of open submanifolds of $M_i$.
\end{thm}

\begin{rmrk}\label{flat-to-GH-r}
If in addition it is assumed that $\lim_{i\to \infty}\mass(M_i)=\mass(M)$, then
by (\ref{flat-to-GH-1}),  
\be
\lim_{i\to \infty} \mass(\set(T_i-S_i), d_i, T_i-S_i)=0.
\ee
In the Riemannian setting, 
\be
\lim_{i\to \infty} \vol(M_i\setminus N_i)=0.
\ee
% this is necessary addl assumption as part of the space may have cancellation.
\end{rmrk}

\begin{rmrk}
In Ilmanen's example \cite{SorWen2} of a sphere with
increasingly many splines, the $S_i$ may 
be chosen to be integration over the spherical
part of $M_i$ with balls around the tips removed.  Then $\set(S_i)$ are manifolds with
boundary converging to the sphere in the Gromov-Hausdorff and intrinsic flat sense.
\end{rmrk}

\begin{rmrk}
The precompactness of the limit integral current spaces is
necessary in this theorem 
because a noncompact limit space can never be the Gromov-Hausdorff
limit of precompact spaces.  In fact there are sequences of
compact Riemannian manifolds, $M_j$, whose intrinsic flat limit
is an unbounded complete Riemannian manifold of finite volume
\cite{SorWen2}[Ex A.10] and another example of
such spaces whose Intrinsic Flat limit is a bounded noncompact
integral current space \cite{SorWen2}[Ex A.11].
\end{rmrk}

\begin{rmrk} \label{rmrk-conv-point}
Gromov's Compactness Theorem combined with Theorem~\ref{flat-to-GH} implies that
that any sequence of $x_i\in N_i\subset M_i$
has a subsequence converging to a point $x$ in the metric completion of $M$.
Other points need not have limit points, as can be seen when the tips of
thin splines disappear in the examples from \cite{SorWen2}.
A more general Bolzano-Weierstrass Theorem precisely
identifying those points which do not disappear is proven
later in this section.
\end{rmrk}

Theorem~\ref{flat-to-GH} is now proven:

\begin{proof}

By Theorem~\ref{converge} there exists a common metric space
$Z$ and isometric embeddings $\varphi_i: X_i \to Z$ and $\varphi:X \to Z$ such that
\be \label{flat-to-GH-1a}
\varphi_{\#}T-\varphi_{i\#}T_i=U_i +\partial V_i
\ee
where $m_i=\mass\left(U_i\right)+\mass\left(V_i\right)\to 0$.
So $\varphi_{i\#}T_i$ converges in the flat and thus the weak sense to $\varphi_\# T$.   

Since $M\in \mathcal{M}^m_0$, $\varphi\left(X\right)$ is precompact. Let
$\rho: Z \to \R$ be the distance function from $\varphi\left(X\right)$.

By the Ambrosio-Kirchheim Slicing Theorem [Theorem~\ref{theorem-slicing}] applied to $f(s)=-\rho(s)$, one has
\be \label{sie-1}
S_{i,\epsilon}:=\varphi_{i\#}T_i \rstr \rho^{-1}\left([0,\epsilon)\right) \in \intcurr_m\left(Z\right)
\ee
for almost every $\epsilon>0$.  Fix any such $\epsilon$.

Before choosing the $S_i$ mentioned in the statement of the theorem, one may examine the mass of $S_{i,\epsilon}$
and the Hausdorff distance between $\set(S_{i\epsilon})$
and $\varphi(X)$. 
Note that $\varphi_{\#}T = \varphi_\#T \rstr \rho^{-1}[0,\epsilon)$.
So 
\be\label{portion}
||T||(\rho^{-1}[0,\epsilon))=\mass(T).
\ee
By lower semicontinuity of mass one has
\be
\liminf_{i\to\infty}
||\varphi_{i\#}T_i||(\rho^{-1}[0,\epsilon))\ge 
||\varphi_{\#}T||(\rho^{-1}[0,\epsilon)). 
\ee
Combining this with (\ref{sie-1}) and (\ref{portion}) 
and the definition of liminf one has:
\be\label{sie-lower-mass}
\textrm{for almost every } \epsilon>0\,\,\exists N'_\epsilon\in\N \textrm{ such that }
\mass(S_{i\epsilon})\ge \mass(T)-\epsilon \,\,\,\forall i\ge N'_\epsilon.
\ee

To see that the Hausdorff distance between $S_{i,\epsilon}$ and $\varphi(X)$ is small, $d_H^Z(S_{i,\epsilon}, \varphi(X))<2\epsilon$, first immediately observe that
\be \label{eqn-sie-1}
\set(S_{i,\epsilon}) \subset \bar{T}_\epsilon(\varphi(X))
\subset T_{2\epsilon}(\varphi(X)).
\ee
One needs only show
\be \label{eqn-sie-back-2}
\varphi\left(X\right) \subset T_{2\epsilon} \left(\set \left(S_{i,\epsilon}\right) \right) \qquad \forall i \ge N_\epsilon.
\ee
To prove (\ref{eqn-sie-back-2}), 
first note that for any $x\in X$, one can let $\rho_x$ be the distance
function from $\varphi\left(x\right)$.  By 
the lower semicontinuity of mass of open sets one has,
\be
\liminf_{i\to\infty} ||\varphi_{i\#}T_i||(\rho_x^{-1}[0,\epsilon))
\ge ||\varphi_{\#}T||(\rho_x^{-1}[0,\epsilon))>0
\,\,\,\forall \epsilon>0.
\ee
Thus one has
\be 
\textrm{ for almost every } \epsilon>0 \,\,\exists N_{\epsilon,x}
\ge N'_\epsilon 
\,\,s.t.\,\, \varphi_{i\#}T_i \rstr \rho_x^{-1}[0,\epsilon)\neq 0 \,\,\,\forall i \ge N_{\epsilon,x}.
\ee
Recall $N'_\epsilon$ was defined in (\ref{sie-lower-mass}). 
Combining this with (\ref{sie-1}), and the fact that
\be
\rho_x^{-1}[0,\epsilon)=B(x,\epsilon)\subset \rho^{-1}[0,\epsilon)
=T_\epsilon(\varphi(X))
\ee
one has
\be \label{eqn-sie-here-compact}
\forall x\in X\,\,\,
\textrm{ for almost every } \epsilon>0 \,\,\exists N_{\epsilon,x}
\ge N'_\epsilon \textrm{ and } s_{i,\epsilon,x}\in \set(S_i)\cap B(\varphi(x),r).
\ee
By the precompactness of $X$, there is
a finite $\epsilon$ net, $X_\epsilon=\{x_1,...x_N\}$ on $\varphi\left(X\right)$ 
(i.e. the union of $B(x_i,\epsilon)$ contains $X_\epsilon$).
Define
\be
N_\epsilon=\max\left\{N_{\epsilon, x_j}: x_j\in X_\epsilon\right\} \ge N'_\epsilon
\ee
then 
\be
\forall x\in X\,\exists x_j\in X_\epsilon \,\,s.t.\,\,
\forall i\ge N_\epsilon\,\,
\exists
s_x:=s_{i,\epsilon,x_j} \in \set \left(S_{i,\epsilon}\right)
\textrm{ s.t. } d_Z\left(s_{x},\varphi (x)\right) <2\epsilon.
 \ee
So (\ref{eqn-sie-back-2}) has been proven.

Combining (\ref{eqn-sie-back-2}) with (\ref{eqn-sie-1}), the Hausdorff distance
satisfies
\be \label{eqn-sie-3}
d^Z_H\left(\set \left(S_{i,\epsilon}\right), \varphi\left(X\right)\right) \le 2\epsilon \qquad \forall i\ge N_\epsilon.
\ee

Recall the definition of $S_i$ as in the statement of the theorem.
One must prove (\ref{flat-to-GH-0}) and (\ref{flat-to-GH-1}).

Let $\epsilon_k\to 0$ be a decreasing sequence of $\epsilon$ for which
all these currents are defined. Let $N_k:=N_{\epsilon_k}$. Let
\be
S_i = T_i \in \intcurr_m\left(X_i\right) 
\textrm{ for } i=1 \textrm{ to } N_1
\ee
\be
S_i= \varphi^{-1}_{i\#}S_{i, \epsilon_1} \in \intcurr_m\left(X_i\right) 
 \textrm{ for } i=N_1+1 \textrm{ to } N_2
\ee
and so on:
\be \label{F-to-GH-S_i}
S_i= \varphi^{-1}_{i\#}S_{i, \epsilon_j} \in \intcurr_m\left(X_i\right) 
\textrm{ for } i=N_j+1 \textrm{ to } N_{j+1}
\ee
Then by  (\ref{eqn-sie-3}), 
\be
d^Z_{H}\left(\set \left(S_{i}\right), \varphi\left(X\right)\right) \le 2\epsilon_i.
\ee
This implies (\ref{flat-to-GH-0}).

By (\ref{sie-lower-mass}) and $N_{k}=N_{\epsilon_k}\ge N'_{\epsilon_k}$ one has,
one has
\be
  \mass(S_{i})\ge \mass(T)-\epsilon_i
  \ee
which gives us (\ref{flat-to-GH-1}) and completes the proof of the theorem.
\end{proof}

\vspace{.2cm}

\begin{rmrk}
One could construct a common metric space $Z$ for
Examples A.10 and A.11 of \cite{SorWen2}
and find $S_{i,\epsilon}$ as in the above proof satisfying 
(\ref{eqn-sie-1}).  However, in that example, (\ref{eqn-sie-back-2}) will fail to hold.
This is where the precompactness of the limit space is essential in the proof.
\end{rmrk}

\vspace{.2cm}

\begin{rmrk} \label{rmrk-region-length}
Examples in \cite{SorWen2} demonstrate 
that the metric space of a current space need not be a length
space.   In general, when a sequence of Riemannian manifolds converges in the intrinsic flat sense to an integral current space
it need not be a geodesic length space.  
If the $\set (S_i)$ are length spaces or approximately length spaces, then the limit current space is in fact a length space.  
This occurs for example in Ilmanen's example of \cite{SorWen2}.
It also occurs whenever the Gromov-Hausdorff limits and flat limits of length spaces agree.
It might be interesting to develop a notion of an approximate length space that suffices to
give a geodesic limit space.  What properties must hold on $M_i$ to guarantee that their
limit is a geodesic length space?
\end{rmrk}

\vspace{.2cm}

\begin{rmrk}
It is not immediately clear whether the integral current spaces, $N_i$,
constructed in the proof of Theorem~\ref{flat-to-GH} actually converge
in the intrinsic flat sense to $M$.  One expects an extra assumption
on total mass would be needed to interchange between flat and weak
convergence, but even so it is not completely clear.  One would need
to uniformly control the masses of $\partial N_i$ using a common upper
bound on $\mass(N)$ which can be done using theorems in 
Section 5 of \cite{AK}, but is highly technical.  It is only worth investigating
if one has an application in mind.
\end{rmrk}

\section{Arzela-Ascoli Theorem for Lipschitz Functions}

In this section the first Arzela-Ascoli Theorem is proven.
This basic theorem is proven using only Theorem~\ref{converge}
and Lemma~\ref{to-a-limit}.

\begin{thm}\label{Flat-Arz-Asc}
\footnote{This theorem originally appeared with a 
fundamentally different more difficult proof involving
Gromov filling volumes in an early preprint 
version of \cite{Sormani-properties}.   It will not appear 
in any publication of \cite{Sormani-properties}.}
Fix $K>0$.
Suppose $M_i=(X_i, d_i, T_i)$ are integral current spaces for
$i\in \{1,2,...,\infty\}$ and
$M_i \Fto M_\infty$
and $F_i: X_i \to W$ are Lipschitz maps into
a compact metric space $W$ with 
\be
\Lip(F_i)\le K,
\ee
then a subsequence converges to a Lipschitz map
$F_\infty: X_\infty \to W$ with 
\be
\Lip(F_\infty)\le K.
\ee  More
specifically, there exists isometric embeddings 
of the subsequence, $\varphi_i: X_i \to Z$,
such that $d_F^Z(\varphi_{i\#} T_i , \varphi_{\infty\#} T_\infty)\to 0$
and for any sequence $p_i\in X_i$ converging to $p\in X_\infty$,
\be\label{pip}
d_Z(\varphi_i(p_i), \varphi_\infty(p))\to 0, 
\ee
one has converging images,
\be\label{Fip}
d_W(F_i(p_i),F_\infty(p))\to 0.
\ee
\end{thm}

\begin{proof}
By Theorem~\ref{converge}, $\varphi_i:M_i \to Z$
such that $d_F^Z(\varphi_{i\#} T_i , \varphi_\infty T_\infty)\to 0$.

Take any $p_\infty\in X_\infty$.  By Lemma~\ref{to-a-limit}, there
exists $p_i \in X_i$ such that $\lim_{i\to\infty} \varphi_i(p_i)=\varphi_\infty(p_\infty)$.   Their images $F_i(p_i) \in W$ have a subsequence which
converges to some $w\in W$.  Set $F_\infty(p_\infty)=w$.    Recall
that integral current spaces are seperable.   So there is a countable dense subset $X_0\subset X_\infty$.  Thus one may repeat this
process creating subsequences of subsequences for a countable
dense collection of $p\in X_0=X_\infty$.   Diagonalizing, one
obtains the subsequence mentioned in the theorem statement
and a function,
\be
F_\infty: X_0\subset X_\infty \to W.   
\ee
One needs to extend $F_\infty$ to define a limit function from
$X$ to $W$.  Observe
that for all $p, q\in X_0$ there exists $p_i$ and $q_i$ converging to them
such that
\begin{eqnarray}
d_W(F_\infty(p), F_\infty(q))&=&\lim_{i\to \infty} d_W(F_i(p_i), F_i(q_i))\\
&\le&\lim_{i\to \infty} K d_{X_i}(p_i, q_i)\\
&\le&\lim_{i\to \infty} K d_{Z}(\varphi_i(p_i), \varphi_i(q_i))\\
&\le&\ K d_{Z}(\varphi_\infty(p), \varphi_\infty(q))\\
&\le & K d_{X_\infty}(p,q).
\end{eqnarray}
Thus one may extend $F_\infty$ continuously to 
\be
F_\infty: X_\infty \to W
\textrm{ and }\Lip(F_\infty)\le K.   
\ee

Now suppose  $p_i\to p$ as in (\ref{pip})
and proceed to prove (\ref{Fip}).   Assume on the
contrary that there exists a subsequence of $p_i$
also denoted $p_i$ such that
\be
\exists r_0>0 \textrm{ s.t. } d_W(F_i(p_i), F_\infty(p))>r_0.
\ee
By (\ref{pip}), there exists $N_0\in \mathbb{N}$ such that
\be
d_Z(\varphi_i(p_i), \varphi_\infty(p))<r_0/10 \,\,\,\forall i \ge N_0. 
\ee
By the definition of
the continuous extension, there exists $q_j \in X_0$
and there exists $N_1\in \mathbb{N}$ such that
\be\label{inZ1}
d_Z(\varphi_\infty(q_j),\varphi_\infty(p))=d_X(q_j,p)<r_0/(10K)
\,\,\,\forall j\ge N_1
\ee
and
\be\label{inW1}
 d_W(F_\infty(q_j), F_\infty(p))\le 
K d_Z(\varphi_\infty(q_j),\varphi_\infty(p))=K d_X(q_j,p)<r_0/10
\,\,\,\forall j\ge N_1.
\ee

By the definition of $F_\infty: X_0 \to W$, for each fixed $j$,
there exists $q_{j,i} \in X_i$ and $N_j, N'_j \in \mathbb{N}$
with 
\be \label{inZ2}
d_Z(\varphi_i(q_{j,i}), \varphi_\infty(q_j))< r_0/(10K)
\,\,\,\forall i\ge N_j
\ee
and 
\be\label{inW2}
d_W(F_i(q_{j,i}),  F_\infty(q_j))\le r_0/5\,\,\, \forall i \ge N'_j.
\ee

Also $\Lip(F_i)\le K$ implies:
\be\label{inW3}
d_{W}(F_i(p_i), F_i(q_{j,i})) \le Kd_{X_i}(p_i, q_{j,i}).
\ee

Take any $j\ge N_1$ and any $i\ge \max\{N'_j,N_j, N_0\}$.
By (\ref{pip}),(\ref{inZ1}) and (\ref{inZ2}) one has
\begin{eqnarray*}
d_{X_i}(p_i, q_{j,i}) &=& d_Z(\varphi_i(p_i),\varphi_i( q_{j,i})) \\
&\le& d_Z(\varphi_i(p_i),\varphi_\infty(p))
+ d_Z(\varphi_\infty(p), \varphi_\infty(q_j))
+ d_Z(\varphi_i(q_{j,i}), \varphi_\infty(q_j))\\
&\le & 3r_0/(10K) 
\end{eqnarray*}

Combining this with (\ref{inW1}), (\ref{inW2}) and (\ref{inW3}), one
has
\begin{eqnarray*}
d_W(F_i(p_i), F_\infty(p))
&\le &d_W(F_i(p_i), F_i(q_{j,i})) +
d_W(F_i(q_{j,i}), F_\infty(q_j)) + d_W(F_\infty(q_j), F_\infty(p))
\\ 
&\le &Kd_{X_i}(p_i, q_{j,i}) +d_W(F_i(q_{j,i}), F_\infty(q_j))
+  K d_X(q_j,p)\\
&\le & K(3r_0/(10K)) + (r_0/5)+ r_0/10=6r_0/10<r_0,
\end{eqnarray*}
which is a contradiction.
\end{proof}

\begin{rmrk}\label{no-Arz-Asc-IF-to-IF}
Recall that the corresponding Gromov-Hausdorff Arzela-Ascoli
Theorem allows the target spaces to vary as well:
 $F_i: X_i \to W_i$ of Lipschitz maps into
compact metric spaces $W_i$ with $\Lip(F_i)\le K$
where $W_i \GHto W$ and $X_i \GHto X$.  
See for example Grove-Petersen \cite{Grove-Petersen}.   
The corresponding statement allowing both $X_i\Fto X$
and $W_i \Fto W$ is false.  For example,
one may have a sequence of compact connected manifolds, $W_i$, 
which converge in the intrinsic flat sense to a compact metric space, $W$, 
that is not connected \cite{SorWen2}.   In that setting one has a 
sequence of Lipschitz maps which are unit speed geodesics, 
$F_i:[0,1]\to W_i$ where $W_i \Fto W$ with no limiting function
$F:[0,1]\to \bar{W}$.   
\end{rmrk}

\begin{rmrk}\label{Arz-Asc-IF-to-GH}
It should be possible to extend Theorem~\ref{Flat-Arz-Asc}
to sequences $F_i: X_i \to W_i$ of Lipschitz maps into
compact metric spaces $W_i$ with $\Lip(F_i)\le K$
where $W_i \GHto W$ and $X_i \Fto X$ using 
Gromov's Embedding Theorem or the work of Grove-Petersen \cite{Grove-Petersen}.   No applications are known
for such a theorem at this time so there is no need to 
prove this here.
\end{rmrk}

%%%%%%%%%%%%%%%%%%%%%%

\section{Basic Bolzano-Weierstrass Theorem}

In this section, Theorem~\ref{B-W-BASIC} is proven.
Recall Lemma~\ref{lem-ball} states that for almost every 
radius $S(p,r)$ of (\ref{Spr}) is an
integral current space.   Recall also that,
like any integral current space, 
$d_{\mathcal{F}}(S(p,r),{\bf{0}})=0$ iff $S(p,r)={\bf{0}}$ \cite{SorWen1}.   
If one considers a sequence of integral current 
spaces, $M_i$ with points $p_i$, then for almost
every $r$, $S(p_i,r)$ is an integral current space for
all $i$ in the sequence.   In this basic Bolzano-Weierstrass
Theorem one assumes these $S(p_i,r)$ are kept a 
definite distance away from ${\bf{0}}$ where this distance depends upon on the radius.   A different Bolzano-Weierstrass Theorem which involves the Gromov Filling Volume appears in \cite{Sormani-properties}.

\begin{thm}\label{B-W-BASIC}\label{BW-BASIC}
Suppose $M^m_i=(X_i, d_i, T_i)$ are integral current spaces 
%with uniform upper bounds, 
%\be
%\mass(M^m_i)\le V_0 \textrm{ and } \mass(\partial M^m_i)\le A_0.
%\ee
which converge in the intrinsic flat sense to a 
nonzero integral current space 
$M^m_\infty=(X_\infty, d_\infty, T_\infty)$.
Suppose there exists $r_0>0$, a positive function
$h:(0,r_0)\to (0,r_0)$, and a sequence
$p_i \in M_i$ such that for almost every $r\in (0, r_0)$ 
\be 
\liminf_{i\to \infty} d_{\mathcal{F}}(S(p_i,r),0) \ge h(r)>0.
\ee 
Then there exists a subsequence, also denoted $M_i$, such that
$p_{i}$ converges to $p_\infty\in \bar{X}_\infty$.
\end{thm} 

\begin{rmrk}
Note that $M_i$ and $M_\infty$ are not required to be precompact.   The $M_i$ are not required to have uniformly bounded mass or volume. 
The key hypothesis is that the $M_i \Fto M_\infty$ and that $M_\infty$
has finite mass.   For this reason there is not enough room to fit too many balls of mass $h(r)$ in $M_\infty$.   This allows us to produce a converging
subsequence in the style of a classical Bolzano-Weierstrass Theorem.   
\end{rmrk}

\begin{rmrk}
It is possible that $p_\infty\notin X_\infty$ as can be seen by taking 
all the $M_i=M_\infty$ a manifold $M$ with a cusp singularity at $p_\infty$
so that $M_\infty=M\setminus p_\infty$ and $p_i$ a sequence of points 
approaching $p_\infty$.
\end{rmrk}

\begin{proof}
By Theorem~\ref{converge} 
there exists a common metric space
$Z$ and isometric embeddings $\varphi_j: X_j \to Z$ and $\varphi_\infty:X_\infty \to Z$ such that
\be \label{A-j-B-j} 
\varphi_{j\#}T_j-T=\partial B_j + A_j
\ee
where $A_j\in \intcurr_m(Z)$ and $B_j\in \intcurr_{m+1}(Z)$ 
with
\be
\mass(A_j)+\mass(B_j) \to 0
\ee
and where
\be
T=\varphi_{\infty\#}T_\infty\in \intcurr_m(Z). 
\ee

One needs only show that a subsequence of $\varphi_i(p_i)$ is a Cauchy sequence.   Once this is done,
one can apply Lemma~\ref{balls-converge} to the subsequence.   In that lemma, it is shown that a Cauchy sequence,
$p_i$, converges to $p_\infty\in \bar{X}_\infty$ unless
there is a radius $r$ sufficiently small that $S(p_i,r) \Fto \bf{0}$.  Since this is not allowed by the
hypothesis of the theorem being proven, one sees that the subsequence converges to
$p_\infty \in \bar{X}_\infty$ as desired.

So one needs only prove that a subsequence $\varphi_i(p_i)$ converges in $Z$.
This is not immediate because $Z$ is only complete and need not be compact.   

Assume on the contrary that 
\be
\exists \delta >0 \,\,\,s.t.\,\,\, d_Z(\varphi_i(p_i), \varphi_j(p_j))\ge \delta \,\,\, \forall i,j\in \mathbb{N}.
\ee
Let $\rho_i(x)=d_Z(\varphi_i(p_i),x)$, then
for almost every $r\in (0,r_0)\cap (0,\delta/2)$,
\be
\rho_i^{-1}(-\infty,r)\cap \rho_j^{-1}(-\infty,r)=\emptyset \qquad \forall i,j\in \mathbb{N}.
\ee
%Then for any $N\in \mathbb{N}$
%\be
%V_0 \ge \mass(T_i) \ge \sum_{j=1}^N \mass\left(\varphi_{i\#}T_i \rstr \rho_j^{-1}(-\infty,r) \right).
%\ee
Now
\begin{eqnarray}
\,\,\,\,\,\,(\varphi_{i\#}T_i) \rstr \rho_j^{-1}(-\infty,r)& - &\varphi_{\infty\#}T_\infty\rstr \rho_j^{-1}(-\infty,r)=\\
&=&(\partial B_i) \rstr \rho_j^{-1}(-\infty,r)+ A_i\rstr  \rho_j^{-1}(-\infty,r)\\
&=& <B_i, \rho_j, r> + \,\partial \left(B_i \rstr \rho_j^{-1}(-\infty,r)\right)
\\
&&\qquad+\,\,\,A_i\rstr  \rho_j^{-1}(-\infty,r).
\end{eqnarray}
Thus $
d_F^Z\left(\varphi_{i\#}T_i \rstr \rho_j^{-1}(-\infty,r) ,\varphi_{\infty\#}T_\infty\rstr \rho_j^{-1}(-\infty,r)\right)\le$
\begin{eqnarray}
&\le&\,\, f_{ij}(r) + \,\,\,\mass(B_i \rstr \rho_j^{-1}(-\infty,r))
\,\,+\,\, \mass(A_i\rstr  \rho_j^{-1}(-\infty,r)) \\
&\le& \,\,f_{ij}(r) \,\,\,+ \,\,\,\mass(B_i) \,\,\,+\,\,\, \mass(A_i)
\end{eqnarray}
where
\be
f_{ij}(r)=\mass(<B_i, \rho_j, r>).
\ee  
By the Ambrosio-Kirchheim Slicing Theorem, for fixed $j\in \mathbb{N}$,
\begin{eqnarray}
\int_{-\infty}^{\infty} f_{ij}(r) \, dr 
&=&\int_{-\infty}^{\infty} \mass(<B_i, \rho_j, r>) \, dr \\
&=&\quad \mass(B_i\rstr d\rho_j)
\,\,\, \le \,\,\,Lip(\rho_j) \mass(B_i)\,\,\,\le\,\,\, \mass(B_i)
\end{eqnarray}
which converges to $0$ as $i\to \infty$.   Thus for fixed $j$ and almost every 
$r$ there is a subsequence $i'\to \infty$ such that 
$\lim_{i'\to\infty} f_{i'j}(r)=0$ pointwise.  Diagonalizing, there is a subsequence $i"$
such that for all $j$, 
$\lim_{i'\to\infty} f_{i''j}(r)=0$ pointwise.  

Thus for almost every $r\in (0,r_0)\cap (0,\delta/2)$, there is a subsequence $i''$ such that for all $j\in \mathbb{N}$,
\be \label{i''}
d_F^Z\left(\varphi_{i''\#}T_{i''} \rstr \rho_j^{-1}(-\infty,r), \varphi_{\infty\#}T_\infty \rstr \rho_j^{-1}(-\infty,r)\right)\to 0
\ee
Since the balls are disjoint,
\be
\mass(T_\infty)  \ge \sum_{j=1}^N \mass\left(\varphi_{\infty\#}  T_\infty \rstr \rho_j^{-1}(-\infty,r) \right).
\ee
Thus
\be
\limsup_{j\to \infty} \mass\left(\varphi_{\infty\#} T_\infty \rstr \rho_j^{-1}(-\infty,r) \right)=0.
\ee
So
\be
\limsup_{j\to \infty} d_F^Z\left(\varphi_{\infty\#} T_\infty \rstr \rho_j^{-1}(-\infty,r),\bf{0} \right)=0.
\ee
In particular, for $j$ sufficiently large
\be
d_F^Z\left(\varphi_{\infty\#} T_\infty \rstr \rho_j^{-1}(-\infty,r),\bf{0} \right)<h(r)/2.
\ee
Combining this with (\ref{i''}), for $i"$ sufficiently large
\be
d_{\mathcal{F}}\left(S(p_{i"},r), {\bf{0}}\right)\le
d_F^Z\left(\varphi_{i"\#} T_{i"} \rstr \rho_j^{-1}(-\infty,r),\bf{0} \right)<h(r)/2
\ee
which contradicts the hypothesis.
Thus there is a subsequence $\varphi_i(p_i)$ 
which converges to some point $z_\infty \in Z$ exactly as needed.
\end{proof}

%%%%%%%%%%%%%%%%%%%%%%%
\vspace{.4cm}
\section{Limits of Uniformly Local Isometries}
 
In this section as Arzela-Ascoli Theorem which allows both
the domain and the target spaces to converge in the intrinsic
flat sense.   This theorem applies to sequences of oriented Riemannian manifolds $M_i$
with 
\be
\vol(M_i)\le V_i \textrm{ and } \vol(\partial M_i) \le A_i
\ee
and functions $F_i: M_i \to M'_i$ which are orientation preserving local
isometries that are isometries on balls of a fixed radius, $\delta>0$
which is uniform for the sequence.

\begin{thm}\label{Arz-Asc-Unif-Local-Isom}
\footnote{A similar theorem with slightly different hypothesis originally appeared with a 
fundamentally different and more difficult proof involving
Gromov filling volumes in an early preprint version of \cite{Sormani-properties}.   That theorem will not appear 
in any publication of \cite{Sormani-properties} as this
theorem is an improvement.}
Let $M_i=(X_i, d_i, T_i)$
and $M'_i=(X'_i,d'_i,T'_i)$  be integral current spaces 
such that 
\be
M_i \Fto M_\infty \textrm{ and } M'_i \Fto M'_\infty.
\ee

Fix $\delta>0$.
Let $F_i: M_i \to M'_i$ be continuous maps which are current preserving isometries
on balls of radius $\delta$ in the sense that:
\be \label{iso-sat} 
\forall x\in X_i, \,\, F_i: \bar{B}(x,\delta) \to \bar{B}(F_i(x),\delta)\textrm{ is an isometry}
\ee
and
\be\label{curpres}
F_{i\#}(T_i\rstr B(x,r))=T'_i \rstr B(F(x),r) \textrm{ for almost every } r\in (0,\delta).
\ee
Then, when $M_\infty\neq {\bf{0}}$, one has $M'_\infty \neq {\bf{0}}$ and
there is a subsequence, also denoted $F_i$, which
converges to a (surjective) local isometry
\be
F_\infty: \bar{X}_\infty \to \bar{X}'_\infty.
\ee
More
specifically, there exists isometric embeddings
of the subsequence $\varphi_i: X_i \to Z$,
$\varphi'_i: X'_i \to Z'$,
such that 
\be
d_F^Z(\varphi_{i\#} T_i , \varphi_{\infty\#} T_\infty)\to 0 \textrm{ and }
d_F^{Z'}(\varphi'_{i\#} T'_i , \varphi'_{\infty\#} T'_\infty)\to 0
\ee
and for any sequence $p_i\in X_i$ converging to $p\in X_\infty$:
\be
\lim_{i\to\infty} \varphi_i(p_i)=\varphi_\infty(p) \in Z
\ee
one has
\be \label{iso-infty}
\lim_{i\to\infty}\varphi_i'(F_i(p_i))=\varphi_\infty'(F_\infty(p_\infty)) \in Z'.
\ee
When $M_\infty={\bf{0}}$ and $F_i$ are surjective, one has
$M'_\infty={\bf{0}}$.
\end{thm}

\begin{rmrk}\label{r-AA-ULI}
Example~\ref{ex-AA-ULI} describes the necessity of the uniformity
condition (\ref{iso-sat}) in Theorem~\ref{Arz-Asc-Unif-Local-Isom}.  
\end{rmrk}

\begin{rmrk}
It may be possible to prove that the limit map here is
also current preserving on balls of radius less than $\delta$.
This is technical and not needed for present applications but
might be an interesting investigation in the future.
\end{rmrk}

\begin{rmrk}\label{r-biLip}
It may be possible to prove a similar theorem replacing the 
surjective uniformly
local isometries
with surjective uniformly local uniformly bi-Lipschitz maps but the proof would
be fairly technical and there is no immediate application for this
at this time.  
\end{rmrk}

Theorem~\ref{Arz-Asc-Unif-Local-Isom} is now proven:

\begin{proof}
%First observe that since $F_i$ are surjective local current preserving isometries, we have
%\be\label{mass-control-AA}
%\mass(M_i')\le \mass(M_i)\le V_0 \textrm{ and }\mass(\partial M'_i)\le \mass(\partial M_i) \le A_0
%\ee
By Theorem~\ref{converge} there exists $\varphi_i:M_i \to Z$
such that 
\be
d_F^Z(\varphi_{i\#} T_i , \varphi_{\infty\#} T_\infty)\to 0
\ee
and
$\varphi'_i:M'_i \to Z'$
such that 
\be
d_F^{Z'}(\varphi'_{i\#} T'_i , \varphi'_{\infty\#} T'_\infty)\to 0.
\ee

Assuming $M'_\infty\neq {\bf{0}}$,
one must first find a subsequence and construct the limit function 
$F_\infty: P \to X'_\infty$ satisfying (\ref{iso-infty}) for all $p\in P$
where $P$ is a countably dense collection of points in $X_\infty$.

Take any $p\in P$.  Recall
$S(p,r)=(\set(T_\infty \rstr B(p,r)), d_\infty,T_\infty \rstr B(p,r))$ is defined for almost every $r$.
Since $p\in X_\infty$, and $X_\infty=\set(T_\infty)$,
\be
\liminf_{r\to 0}\mass(S(p,r)) / r^m =
\liminf_{r\to 0}||T_\infty||(B(p,r)) / r^m 
>0.
\ee
In particular 
\be \label{r_p}
S(p,r) \neq {\bf{0}}. 
\ee

By Lemma~\ref{to-a-limit} there exists
$p_i \in X_i$ such that 
\be
\lim_{i\to\infty} \varphi_i(p_i)=\varphi_\infty(p).
\ee
By Lemma~\ref{balls-converge}, for almost every $r_\infty>0$,
there is a subsequence (also denoted $i$) such that
\be
d_{\mathcal{F}}\left(S(p_i,r_\infty),S(p,r_\infty)\right)\to 0.
\ee
Taking $r_\infty=\delta$, applying (\ref{curpres}) one has
\be
F_{i\#}S(p_i,r_\infty)=S(p'_i,r_\infty) \textrm{ where } p'_i=F_i(p_i)
\ee
so
\be
d_{\mathcal{F}}\left(S(p'_i,r_\infty),S(p,r_\infty)\right)\to 0.
\ee
Combining via the triangle inequality with (\ref{r_p}), 
\be
\liminf_{i\to\infty}d_{\mathcal{F}}\left(S(p'_i,r_\infty),{\bf{0}}\right)>0.
\ee
Thus applying the basic Bolzano-Weierstrass Theorem [Theorem~\ref{B-W-BASIC}]
to $S(p_i',r_\infty)$, to see there is a $p_\infty \in \bar{X}'_\infty$ and a further subsequence (also denoted $i$) 
such that $p'_i \to p'_\infty$ in the sense that
\be
\varphi'_i(p'_i) \to \varphi'_\infty(p'_\infty) \in Z'.
\ee
Define $F_\infty(p)=p_\infty$.   

Repeat this process to choose subsequences and $p_\infty$ for each $p$ in the countable collection 
$P\subset  X_\infty$.   Diagonalize to obtain the subsequence
in that statement of the theorem (also denoted $M_i$). 
Thus $F: P \to \bar{X}'_\infty$ is defined such that
\be
\varphi_\infty(F_\infty(p))=\lim_{i\to\infty} \varphi'_i(F_i(p_i))\in Z'.
\ee
To see that $F$ is distance preserving
for any $p,q$ in a ball of radius $\delta$ in $X_\infty$:
\begin{eqnarray}
d_{\bar{X}_\infty}\left(F_\infty(p), F_\infty(q)\right)
&=& d_{Z'}\left(\varphi_\infty(F_\infty(p)),\varphi_\infty(F_\infty(q))\right)\\
&=&\lim_{i\to\infty} d_{Z'}\left(\varphi'_i(F_i(p_i)), \varphi'_i(F_i(q_i)) \right)\\
&=&\lim_{i\to\infty} d_{Z}\left(\varphi_i(p_i), \varphi_i(q_i) \right)\\
&=&d_Z(\varphi_\infty(p),\varphi_\infty(q))=d_{X_\infty}(p,q).
\end{eqnarray}
In particular $F: P \to \bar{X}'_\infty$ is continuous and can be extended to the metric completion,
$F_\infty: \bar{X}_\infty \to \bar{X}'_\infty$ which is an isometry on balls of radius $\delta$.

For any sequence $q_i \in X_i$ converging to $q\in X_\infty$ 
one must show $F_i(q_i)$ converges to $F(q)$.   Assume on
the contrary that this fails:
\be\label{contra-here}
\exists r_0>0 \,\,\exists N_0\in \mathbb{N}\,\,\textrm{ s.t.}\,\,
d_{Z'}(\varphi'_i(F_i(q_i)), \varphi'_\infty(F_\infty(q)))>r_0.
\ee
Since $q_i\to q$, there is an $N_0$ 
sufficiently large that
\be\label{N-0}
d_{X_\infty}(\varphi_i(q_i),\varphi_\infty(q))\le r_0/10.
\qquad \forall i\ge N_0.
\ee
Take $x_j \in P\subset X_\infty$ converging to $q$, and $N_1$ large
enough that
\be\label{N-1}
d_{Z}(\varphi_\infty(x_j)), \varphi_\infty(q)<r_0/10 \qquad \forall j\ge N_1.
\ee
For each $i$, take $x_{j,i} \in X_i$ converging to $x_j$ such that $F_j(x_{j,i})\to F_\infty(x_j)$.   That is there exists $N_i$
and $N'_i$ sufficiently large that
\be\label{N-i}
d_Z(\varphi_i(x_{j,i}),\varphi_\infty(x_j)) < r_0/10 \,\, \forall j \ge N_i
\ee
and
\be\label{N'-i}
d_{Z'}\left(\varphi'_i(F_i(x_{j,i})),\varphi'_\infty(F_\infty(x_j))\right) 
< r_0/10 \,\, \forall j \ge N'_i.
\ee
Since $F_i$ and $F_\infty$ are local isometries
both are distance nonincreasing.  In addition
one has $F_i \circ \varphi_i=\varphi_i'\circ F_i$.
Thus one has for
$i \ge N_0$ and $j\ge \max\{N_1, N_i, N'_i\}$,
\begin{eqnarray*}
d_{Z'}(\varphi'_\infty(F_\infty(q)),\varphi'_i(F_i(q_i)))
&\le &
d_{Z'}(\varphi'_\infty(F_\infty(q)), \varphi'_\infty(F_\infty(x_j))    )\\
&&\qquad+\,\, d_{Z'}(\varphi'_\infty(F_\infty(x_j)), \varphi'_i(F_i(x_{j,i})) )\\
&&\qquad \qquad
+\,\,d_{Z'}(  \varphi'_i(F_i(x_{j,i}))    ,\varphi'_i(F_i(q_i)))\\
&\le &
d_{X'_\infty}(F_\infty(q), F_\infty(x_j))
+ r_0/10 \\
&&\qquad +\,\,
d_{X'_i}(  F_i(x_{j,i}),F_i(q_i))
\,\,\,\textrm{ by } j\ge N'_i,\\
&\le &
d_{X_\infty}(q,x_j)
+ r_0/10 + 
d_{X_i}( x_{j,i},q_i) 
\\
&\le &
r_0/10+ r_0/10 + d_{Z}(\varphi_i( x_{j,i}),\varphi_i(q_i))
\textrm{ by } j \ge N_1, 
\\
&\le &
r_0/5+ d_{Z}(\varphi_i( x_{j,i}),\varphi_\infty(x_j))
+d_{Z}(\varphi_\infty(x_j), \varphi_\infty(q))\\
&&\qquad\,\,\,+\,\,\,d_{Z}(\varphi_\infty(q),\varphi_i(q_i))\\
&\le &
r_0/5+ d_{X_i}( x_{j,i}, x_j)
+d_{X_\infty}(x_j, q)
+ r_0/10 \textrm{ by } i\ge N_0,\\
&\le &
r_0/5+ r_0/10
+r_0/10
+r_0/10 \textrm{ by } j\ge N_i,
\end{eqnarray*}
which contradicts (\ref{contra-here}).

%I HAVE NOT YET SHOWN (\ref{curpres}) holds on the limit.
%NOR HAVE I CLAIMED THIS.  IS IT NEEDED IN THE
%WORK WITH BASILIO?

To see that $F_\infty$ is surjective when $F_i$ are surjective, take any $x\in X'_\infty$.
 so
\be \label{onto-start}
\liminf_{r\to 0}\mass(S(x,r)) / r^m >0.
\ee
In particular 
\be\label{r_x}
\exists r_x >0 \,\,s.t.\,\,\,S(x,r) \neq {\bf{0}} \qquad a.r. \,\, r<r_x.
\ee
By Lemma~\ref{to-a-limit} there exists
$x_i \in X'_i$ such that 
\be
\lim_{i\to\infty} \varphi'_i(x_i)=\varphi'_\infty(x)
\ee
and by Lemma~\ref{balls-converge}, for almost every $r>0$
there is a subsequence (also denoted $i$) such that
\be
d_{\mathcal{F}}(S(x_i,r),S(x,r))\to 0.
\ee
Since $F_i$ are surjective, there exists $p_i\in X_i$ such that $F_i(p_i)=x_i$.
However, for almost every $r<\delta$,
\be
F_{i\#}S(p_i,r)=S(x_i,r) 
\ee
so
\be
d_{\mathcal{F}}(S(p_i,r),S(x,r))\to 0.
\ee
and
\be \label{onto-end}
\liminf_{i\to\infty}d_{\mathcal{F}}(S(p_i,r),{\bf{0}})=h>0.
\ee
Thus applying the basic Bolzano-Weierstrass Theorem [Theorem~\ref{BW-BASIC}], there is a further subsequence of the $p_i$ which converges to a $p_\infty \in X_\infty$.   
To see that $F_\infty(p_\infty)=x$ observe that
\begin{eqnarray}
\varphi_\infty(F_\infty(p_\infty))&=&\lim_{i\to\infty} \varphi_i(F_i(p_i))\\
&=&\lim_{i\to\infty} \varphi_i(x_i)=\varphi_\infty(x_\infty).
\end{eqnarray}

Now suppose $M_\infty={\bf{0}}$.   One needs only show that $M'_\infty={\bf{0}}$.  If not there
exists $x\in X'_\infty$ such that (\ref{onto-start})-(\ref{onto-end}) hold.
However by Lemma~\ref{balls-converge}
\be 
\lim_{i\to\infty}d_{\mathcal{F}}(S(p_i,r),{\bf{0}})=0
\ee
which contradicts (\ref{onto-end}).
\end{proof}

\begin{example}\label{ex-AA-ULI}
The hypothesis that a uniform $\delta>0$ exists such that
(\ref{iso-sat}) holds is necessary.  This can be seen by
taking $M_i$ to be standard flat $1\times 1$ tori and
$M'_i$ to be flat $1\times (1/i)$ tori.  Let $F_i: M_i \to M'_i$
be the $i$ fold covering maps which are surjective local 
isometries on balls of radius $\delta_i=1/(2i)$.   Then 
$M_i$ converges in the intrinsic flat sense to a standard flat
torus while $M'_i$ converges in the intrinsic flat sense to
the $\bf{0}$ integral current space.  Thus there cannot be
any limit map $F_\infty$.
\end{example}

\section{Example with no Intrinsic Flat Limit} \label{sect-example}

The theorems in this paper may be applied to prove certain sequences
of Riemannian manifolds do not converge or converge to specific
Riemannian manifolds.   One such example is provided here.  Further examples will appear in joint work with
Basilio \cite{Basilio-Sormani-1}. 

\begin{example}\label{ex-no-limit}
There exists a sequence of smooth Riemannian manifolds with
boundary with constant sectional curvature such that $\vol_{m-1}(\partial M_j)\le A_0$,
$\diam(M_j) \le D_0$ such that no subsequence converges in the
intrinsic flat or Gromov-Hausdorff sense not even to $\bf{0}$.
\end{example}

\begin{proof}
Let $M_j$ be the $j$ fold covering space of 
\be
N_j=\mathbb{S}^2\setminus\left( B_{p_+}(1/j), B_{p_{-}}(1/j)\right)
\ee
where $\mathbb{S}^2$ is endowed with the standard metric tensor $g_{\mathbb{S}^2}$ which is lifted to $M_j$ and
$p_+$ and $p_-$ are opposite poles.   Let $d_j$ be the
length metric on $M_j$ defined by this metric tensor.

Then
\be
\diam(M_j) \le \pi + j 2\pi(1/j) +\pi = 4\pi
\ee
and
\be
\vol_{m-1}(\partial M_j)\le j \vol_{m-1}(\partial N_j)\le j 2 (2\pi/j)=4\pi
\ee
but 
\be
\lim_{j\to \infty}\frac{\vol_m(M_j)}{j}=\lim_{j\to\infty}\vol(N_j)=\vol(\mathbb{S}^2)= 4\pi.
\ee

Suppose on the contrary that
a subsequence converges $M_{j'}\Fto M_\infty$.

Case I: $M_\infty =\bf{0}$.  If so, then by Lemma~\ref{balls-converge},  any sequence $q_j\in M_j$ and almost every
$r>0$, there is a subsequence $S(q_{j''}, r) \Fto \bf{0}$.   
Take $q_j$ lying on
the equator and choose an $r<1/2$.  Then by the convexity of balls
one has
\be
S(q_j,r) = \left(B(p_0,r), d_{\mathbb{S}^2}, \int_{B(p_0,r)} \right)
\ee
are all isometric to one another.   Thus they do not converge to $\bf{0}$ and there is a contradiction.

Case II: $M_\infty \neq \bf{0}$. Let $x_{j,1},x_{j,2},...,x_{j,j}$ lie on the equator of $X_j$
so that
\be
d_{X_j}(x_{j,i},x_{j,k}) \ge  \pi \qquad \forall i,k \in \{1,2,..,j\}.
\ee
Observe also that
$
B(x_{j,k},\pi/4)
$
are disjoint and are all isometric to a ball $B(x,\pi/4)$ in a standard sphere.
Thus
\be
d_{\mathcal{F}}( S(x_{j,k},\pi/4), S(x, \pi/4))=0 \qquad \forall k \in \{1,2,..,j\}.
\ee
and
\be
d_{\mathcal{F}}( S(x_{j,k},\pi/4), {\bf{0}})=
h_0=d_{\mathcal{F}}( S(x, \pi/4), {\bf{0}})>0 \qquad \forall k \in \{1,2,..,j\}.
\ee
Applying Theorem~\ref{BW-BASIC}, there is a subsequence of each
$x_{k,j}$ must converge to some $x_k \in \bar{X}_\infty$.    Diagonalizing, there
is a subsequence (also denoted $M_j$) such that $x_{k,j}\to x_k$ for all $k$:
so that
\be
d_{X_\infty}(x_k, x_{k'})\ge \pi
\ee
so that
$
B(x_{j,k},\pi/4)
$
are disjoint.  Applying Lemma~\ref{balls-converge},
\be
\lim_{j\to\infty}d_{\mathcal{F}}( S(x_{j,k},\pi/4), S(x_k, \pi/4))=0 .
\ee
and so
\be
d_{\mathcal{F}}( S(x_{k},\pi/4), S(x, \pi/4))=0. 
\ee
Thus $M_\infty$ contains infinitely many balls of the same mass, which contradicts the
fact that
$\mass(T_\infty)$ is finite.
\end{proof}

\section{Further Applications}\label{sect-Appl}

In this section a number of applications of
the theorems in this paper are discussed.

\begin{rmrk} \label{r-AA-BIG}
In \cite{Burago-Ivanov-Vol-Tori}, Burago and Ivanov prove that
the volume growth of the universal cover of a Riemannian manifold 
homeomorphic to a torus is at least that of Euclidean space.
If it is exactly equal, then they have a rigidity theorem
stating that the Riemannian manifold is flat.
Theorem~\ref{Arz-Asc-Unif-Local-Isom} may be useful
in the study of questions arising in Gromov's work
\cite{Gromov-Hilbert-I} analyzing
the almost rigidity of Burago-Ivanov's Theorem (where the
volume growth is close to that of Euclidean space).   Examples
related to this question applying Theorem~\ref{Arz-Asc-Unif-Local-Isom} will appear in upcoming work of the author with Jorge Basilio
\cite{Basilio-Sormani-1}.
\end{rmrk}

\begin{rmrk}\label{r-Wei-1}
Theorem~\ref{Arz-Asc-Unif-Local-Isom} 
should be useful when wishing to study limits
of covering maps and analyzing the existence of a universal
cover of an intrinsic flat limit.  Recall that in joint work with
Wei, the author has conducted such an analysis of Gromov-Hausdorff
limits \cite{SorWei1}.   Zahra Sinaei and the author are exploring
applications in this direction in \cite{Sinaei-Sormani-1}
\end{rmrk}

\begin{rmrk}\label{r-Wei-2}
Theorem~\ref{Arz-Asc-Unif-Local-Isom} 
should also be useful when studying 
how covering spectra behave under intrinsic flat 
convergence.   See joint work of the author with
Wei in which it was shown that covering spectra
behave continuously under Gromov-Hausdorff
convergence \cite{SorWei3}.  
\end{rmrk}

\begin{rmrk}\label{r-funct}
Theorem~\ref{Flat-Arz-Asc} may possibly be applied to study the limits of harmonic functions, eigenfunctions and heat kernals.  Recall that
Cheeger-Colding proved the convergence of eigenfunctions and eigenvalues 
when the Riemannian manifolds are converging in the measured Gromov-Hausdorff sense with a uniform lower bound on Ricci curvature \cite{ChCo-PartIII}.   Ding has proved the convergence of heat kernels in this setting
\cite{Ding-heat}.  Building on work of Fukaya \cite{Fukaya-1987}, Sinaei has
proven the convergence of harmonic maps in this setting with additional
conditions \cite{Sinaei-Harmonic}.
Portegies has shown that eigenvalues need not converge 
when one only has intrinsic flat convergence without a volume bound, but
building on work of Fukaya \cite{Fukaya-1987} has shown the eigenvalues 
semiconverge as long as the volume converges \cite{Portegies-F-evalue}.
It would be interesting to examine what happens to the eigenfunctions
and heat kernels in this setting.   %Sinaei, Portegies?
\end{rmrk}

\begin{rmrk}\label{r-Ricci}
Recall that in \cite{SorWen1}, the author and Wenger proved that
intrinsic flat and Gromov-Hausdorff limits agree when the sequence
of manifolds has nonnegative Ricci curvature and no boundary.  In
particular, there are no disappearing sequences of points in this setting.   In that paper Gromov filling volumes and a contractibility theorem of Perelman was required to complete the argument.
Theorems~\ref{Flat-Arz-Asc} and~\ref{BW-BASIC} should allow one  to prove that there are no
disappearing points as long as the manifolds have a uniform lower bound
on Ricci curvature without resorting to the work of Gromov or Perelman.   A recent preprint by Munn applies theorems from this paper to address this question \cite{Munn-F=GH}.
\end{rmrk}

\begin{rmrk}\label{Lee-Sormani}
In joint work with Dan Lee \cite{LeeSormani1}, 
it has been conjectured that sequences of manifolds with 
nonnegative scalar curvature and no interior minimal surfaces
whose ADM mass converges to $0$ must converge in the
pointed intrinsic flat sense to Euclidean space.  The 
conjecture is proven in that paper
in the rotationally symmetric case.   Lan-Hsuan Huang
and Dan Lee have results working towards proving this
conjecture whenever the manifolds are graphs \cite{Huang-Lee-Graph}.  Their work combined with the results in this paper 
may help complete the proof of this conjecture in the graph case.
More precisely, one should be able to apply Theorem~\ref{Flat-Arz-Asc} to the natural embedding maps from the manifolds to their
images as graphs in Euclidean space and then use theorems
of Huang-Lee regarding the limits of those graphs.
\end{rmrk}

\begin{rmrk}\label{new-examples}
It can be very difficult to prove a sequence of manifolds converges
in the intrinsic flat sense to a particular limit.   In the original
paper introducing intrinsic flat convergence \cite{SorWen2}, 
the author and Stefan Wenger had to construct sequences of
filling manifolds explicitly to prove these examples converged.  In
joint work with Sajjad Lakzian, theorems were proven to allow
one to construct intrinsic flat limits as long as the manifolds were
converging smoothly on sufficiently nice subregions \cite{Lakzian-Sormani}.   Additional such theorems were proven by
Lakzian in \cite{Lakzian-Diameter} and applied to 
Ricci flow through singularities by Lakzian in \cite{Lakzian-Cont-Ricci}.    Theorem~\ref{BW-BASIC} may now be applied to
prove sequences converge in the intrinsic flat sense to limits even
when there is no smooth convergence anywhere.   In joint work
of the author with Jorge Basilio
\cite{Basilio-Sormani-1}, Theorem~\ref{BW-BASIC} is applied
to prove a collection of examples of sequences of manifolds
with nonnegative scalar curvature that have surprising limits.
\end{rmrk}

\begin{rmrk}\label{Sobolev}
In joint work with LeFloch \cite{LeFloch-Sormani-1} it is
proven that sequences of rotationally symmetric
regions with nonnegative scalar curvature, no interior
minimal surfaces and uniformly bounded Hawking mass
have subsequences which converge in the Intrinsic Flat sense.
The proof consists of first proving a Sobolev limit of the
metric tensors exist for a well chosen gauge and then
showing the sequence converges in the intrinsic flat
sense to the Sobolev limit.   In order to extend this relationship
between Sobolev limits and intrinsic flat limits to the
nonrotationally symmetric setting, one may try to apply
theorems from this paper in the same way that they are being
applied as described in Remark~\ref{new-examples}.
\end{rmrk}

\begin{rmrk}\label{cosmos}
In early work, the author studied the stability of the spacelike Friedmann
model of cosmology using the Gromov-Hausdorff distance \cite{Sor-cosmos}.   The Arzela-Ascoli Theorem for Gromov-Hausdorff 
convergence was a key ingredient in this work.
In order to apply Gromov-Hausdorff convergence, one
could not allow the universes under consideration to develop
thin deep wells.  However in work with Dan Lee \cite{LeeSormani1},
it is seen that thin deep gravity wells naturally occur even in
regions of small mass.   In order to study the stability of the 
spacelike Friedmann model of cosmology in a way which permits
thin deep gravity wells, one needs to use the intrinsic flat distance
(otherwise there are counterexamples).    The new Arzela-Ascoli
Theorems provided in this paper should now allow one to
extend the techniques in \cite{Sor-cosmos} to prove a new
intrinsic flat stability theorem for the spacelike Friedmann model
which allows for gravity wells.
\end{rmrk}

If a reader is interested in studying any of these questions, please contact the author.   More details can be provided and the author can coordinate the research of those working on these problems.  Funding to visit the author may be available.

\bibliographystyle{alpha}
\bibliography{2013}

%\bibliography{SorWen2}
\end{document}